\documentclass[11pt,english]{smfart} 
\usepackage[T1]{fontenc}
\usepackage[english,french]{babel}
\usepackage{amscd}
\usepackage{url,xspace,smfthm}
\usepackage{bbm}      
\usepackage{amssymb,latexsym}
\usepackage{amsbsy,bm} 
\usepackage{paralist}
\usepackage{datetime} 
\usepackage{caption}
\usepackage{subcaption} 
\usepackage{colortbl}
\usepackage{color}
\usepackage[dvipsnames]{xcolor}
\usepackage{mathrsfs}  
\usepackage[colorlinks,	colorlinks,
linkcolor=BrickRed,
citecolor=blue,
urlcolor=Cerulean,hypertexnames=true]{hyperref} 
\usepackage{mathtools}
\usepackage{euscript}
\usepackage[all]{xy}
\usepackage{hyperref}
\usepackage{graphicx} 
\usepackage[text={150mm,251mm},centering, marginparwidth=75pt]{geometry} 
\usepackage{comment}  
\usepackage{cite}  
\usepackage{thmtools}
\usepackage{enumitem}
\usepackage{letltxmacro} 
\usepackage{nameref}
\usepackage{cleveref}
\usepackage{stmaryrd}
\usepackage{calc}
\usepackage{interval}
\usepackage{manfnt}
\usepackage{tikz-cd} 
\usepackage[utf8]{inputenc}
\usepackage{marginnote}
\usepackage{smfhyperref}
\usepackage[hyperpageref]{backref}
\renewcommand{\backref}[1]{$\uparrow$~#1}

\usepackage{faktor} 
\usepackage{scalerel}

\makeatletter     
\newcommand{\crefnames}[3]{%
	\@for\next:=#1\do{%
		\expandafter\crefname\expandafter{\next}{#2}{#3}%
	}%
}
\makeatother

\setlist[itemize]{wide = 0pt, labelwidth = 2em, labelsep*=0em, itemindent = 0pt, leftmargin = \dimexpr\labelwidth + \labelsep\relax, noitemsep,topsep = 1ex,}
\setlist[enumerate]{wide = 0pt, labelwidth = 2em, labelsep*=0em, itemindent = 0pt, leftmargin = \dimexpr\labelwidth + \labelsep\relax, noitemsep,topsep = 1ex}

\theoremstyle{plain}

\renewcommand{\thethmx}{\Alph{thmx}} 
\newtheorem{theorem}{Theorem}[section]

\newtheorem{proposition}[theorem]{Proposition}
\newtheorem{corollary}[theorem]{Corollary}

\theoremstyle{definition}
\newtheorem{definition}[theorem]{Definition} 
\theoremstyle{remark}
\newtheorem{remark}[theorem]{Remark} 
\newtheorem{example}[theorem]{Example}

\newtheorem{question}[theorem]{Question}

\numberwithin{equation}{section}  

\theoremstyle{plain}
\newlist{thmlist}{enumerate}{1}
\setlist[thmlist]{wide = 0pt, labelwidth = 2em, labelsep*=0em, itemindent = 0pt, leftmargin = \dimexpr\labelwidth + \labelsep\relax, noitemsep,topsep = 1ex, font=\normalfont, label=(\roman*), ref=\thetheorem.(\roman{thmlisti})}

\addtotheorempostheadhook[theorem]{\crefalias{thmlisti}{thm}}

\addtotheorempostheadhook[assumption]{\crefalias{thmlisti}{assumption}}

\addtotheorempostheadhook[corollary]{\crefalias{thmlisti}{cor}}

\addtotheorempostheadhook[proposition]{\crefalias{thmlisti}{proposition}}

\addtotheorempostheadhook[definition]{\crefalias{thmlisti}{dfn}}

\addtotheorempostheadhook[example]{\crefalias{thmlisti}{example}}

\addtotheorempostheadhook[lemma]{\crefalias{thmlisti}{lem}}
\addtotheorempostheadhook[main]{\crefalias{thmlisti}{main}}

\addtotheorempostheadhook[remark]{\crefalias{thmlisti}{rem}}

\newlist{thmenum}{enumerate}{1} 
\setlist[thmenum]{wide = 0pt, labelwidth = 2em, labelsep*=0em, itemindent = 0pt, leftmargin = \dimexpr\labelwidth + \labelsep\relax, noitemsep,topsep = 1ex, font=\normalfont, label=(\roman*), ref=\thethmx.(\roman{thmenumi})}
\crefalias{thmenumi}{thmx} 

\newlist{corlist}{enumerate}{1} 
\setlist[corlist]{wide = 0pt, labelwidth = 2em, labelsep*=0em, itemindent = 0pt, leftmargin = \dimexpr\labelwidth + \labelsep\relax, noitemsep,topsep = 1ex, font=\normalfont, label=(\roman*), ref=\thecorx.(\roman{corlisti})}
\crefalias{corlisti}{corx} 

\crefname{lemma}{Lemma}{Lemmas} 
\crefname{conjecture}{Conjecture}{Conjectures}
\crefname{theorem}{Theorem}{Theorems}
\crefname{proposition}{Proposition}{Propositions}
\crefname{example}{Example}{Examples}
\crefname{figure}{Figure}{Figures}
\crefname{definition}{Definition}{Definitions}
\crefname{remark}{Remark}{Remarks}
\crefname{corollary}{Corollary}{Corollaries}
\crefname{corx}{Corollary}{Corollaries}
\crefname{problem}{Problem}{Problems}
\crefname{thmx}{Theorem}{Theorems}
\crefname{claim}{Claim}{Claims}
\crefname{assumption}{Assumption}{Assumptions}
\crefname{main}{Main Theorem}{Main Theorems}
\crefname{question}{Question}{Questions}

\newcommand{\C}{\mathbb{C}}
\newcommand{\R}{\mathbb{R}}
\newcommand{\Q}{\mathbb{Q}}

\newcommand{\cJ}{\mathcal{J}}
\newcommand{\sV}{\mathcal{V}}

\newcommand{\orb}{\mathrm{orb}}

\newcommand{\bbmk}{\mathbbm{k}}
\newcommand{\K}{\mathbb{K}}
\newcommand{\Z}{\mathbb{Z}}

\newcommand{\Fp}{\mathbb{F}_p}
\newcommand{\IP}{\mathfrak{p}}
\newcommand{\Ann}{\mathrm{Ann}}

\newcommand{\bF}{\mathbb{F}}

\newcommand{\rank}{{\mathrm{rank}}}
\newcommand{\bw}{{\textbf{w}}}
\newcommand{\Trop}{{\mathrm{Trop}}}
\newcommand{\trop}{{\mathrm{trop}}}

\newcommand{\lcm}{{\mathrm{lcm}}}
\newcommand{\Hom}{{\mathrm{Hom}}}

\begin{document}

\title[BNSR invariants and integral tropical varieties]{Bieri-Neumann-Strebel-Renz invariants and tropical varieties of integral homology jump loci}

\author{Yongqiang Liu}
\address{Institute of Geometry and Physics, University of Science and Technology of China, Hefei 230026, P.R. China} 
\email{liuyq@ustc.edu.cn}

\author{Yuan Liu}
\address{Institute of Geometry and Physics, University of Science and Technology of China, Hefei 230026, P.R. China}
\email{yuanliu@ustc.edu.cn}

\date{\today}
\begin{abstract}
Papadima and Suciu studied the relationship between the Bieri-Neumann-Strebel-Renz (short as BNSR) invariants of spaces and the homology jump loci of rank one local systems. Recently, Suciu improved these results using the tropical variety associated to the  homology jump loci of complex rank one local systems. In particular, the translated positive-dimensional component of homology jump loci can be detected by its tropical variety. In this paper, we generalize Suciu's results to  integral coefficients and give a better upper bound for the BNSR invariants. Then we provide applications mainly to K\"ahler groups. Specifically, we classify the K\"ahler group contained in a large class of 
groups, which we call the weighted right-angled Artin groups. This class of groups comes from the edge-weighted finite simple graphs and is a natural generalization of the right-angled Artin groups. 
\end{abstract}
\maketitle

\section{Introduction}
\subsection{Background}
In 1987, a powerful group theoretic invariant was introduced by Bieri, Neumann and Strebel in \cite{BNS}, now called the BNS invariant. This invariant is a generalization of a former invariant studied by Bieri and Strebel in \cite{BS80,BS81} for metabelian groups. The BNS invariant was later generalized to higher degrees for groups by Bieri and Renz \cite{BR88} and from groups to spaces by  Farber, Geoghegan, and Sch\"utts  in \cite{FGS}. These invariants are called the Bieri-Neumann-Strebel-Renz (short as BNSR) invariants, which record the geometric finiteness properties of  the spaces. 

The computation of the BNSR invariant is extremely difficult. Even in degree $1$ case, it is only known for restricted types of groups, such as metabelian groups \cite{BS80,BS81,BG84}, one relator groups \cite{Brown}, right-angled Artin groups \cite{MV,MMV}, K\"ahler groups \cite{Delzant10} and pure braid groups \cite{KMM15}, etc.   Papadima and Suciu  in \cite{PapaSuciu10} initiated the project of looking for approximations of the BNSR invariants, which (1) are more computable and (2) are rationally defined upper bounds for the BNSR invariants. These bounds are derived from the homology jump loci, defined using the homology of the space with field coefficients in rank one local systems.  Recently, Suciu improved this bound in \cite{Suciu21} using the tropical variety associated to the  homology jump loci of  rank one local systems with complex coefficients. In particular, the translated positive-dimensional component of homology jump loci can be detected by its tropical variety. 

In this paper, we follow Suciu's approach in \cite{Suciu21}  and study the tropical varieties of homology jump loci with integral coefficients. The complement of these tropical varieties gives better upper bounds for the BNSR invariants.

\subsection{Main results}
Let $X$ be a connected finite CW complex with $\pi_1(X)=G$. Let $\mathrm{S}(G)$ denote the unit sphere in the real vector space $\mathrm{Hom}(G;\R) \cong H^1(X; \R)$. In this paper, \textit{we always assume that $\dim H^1(X; \R)>0$}. Set $H=H_1(X; \Z)$, which is the abelianization of $G$. Then it is clear that $\mathrm{S}(G)=\mathrm{S}(H)$.
We say $\chi \in 
\mathrm{S}(G)$ is rational if the image of $\chi$ is isomorphic to $\Z$.
For any integer $k\geqslant 0$, the $k$-th BNSR invariant $\Sigma^k(X; \Z)$ (see \cref{def BNSR}) forms a decreasing sequence of open subsets of $\mathrm{S}(G)$ as $k$ increases. 

Let $\bbmk$ be a coefficient field.  The homology jump ideal $ \cJ^{\leqslant k}(X; \Z)$ (resp. $ \cJ^{\leqslant k}(X; \bbmk)$) 
can be defined via the cellular chain complex of the \textit{maximal abelian cover} of $X$ with coefficients in $\Z$ (resp. $\bbmk$), as a complex of $\Z H$ (resp. $\bbmk H$) modules (see \cref{defn:jump ideal}). In fact, $ \cJ^{\leqslant k}(X; \Z)$ (resp. $ \cJ^{\leqslant k}(X; \bbmk)$)
is an ideal in $\Z H$ (resp. $\bbmk H$).
When $\bbmk$ is an algebraically closed field, the variety of the ideal $ \cJ^{\leqslant k}(X; \bbmk)$ is 
exactly the homology jump loci $\sV^{\leqslant k}(X;\bbmk)$, i.e, the collection of the rank one $\bbmk$-coefficient local systems on $X$ such that its homology is non-zero for some degree in the range $[0,k]$ (see \cref{def homology jump loci}). 
We refer the readers to Suciu's survey paper \cite{Suciu09} for a comprehensive background on this topic.

For any ideal $I \subset \Z H$ (resp. $\bbmk H$), 
one can define its tropicalization $\Trop_\Z (I)$  (resp. $\Trop_\bbmk(I)$)  
in $\mathrm{Hom}(H;\R)$. Since tropical varieties over $\Z$ are relatively uncommon, we provide a detailed study in \cref{sec:tropical}. 
For any subset $Z\subseteq \mathrm{Hom}(H ;\R)$, denote the image of $Z-\{0\}$ in $\mathrm{S}(H)$ under natural projection as $\mathrm{S}(Z)$. Our main result reads as follows.

\begin{theorem}\label{thm:compare_two_bnsr}
With the above notations and assumptions, we have
\begin{equation} \label{main inclusion}
   \Sigma^k(X;\Z)\subseteq \mathrm{S}\big(\Trop_{\Z}(\mathcal{J}^{\leqslant k}(X; \Z))\big)^c\subseteq \mathrm{S}\big(\Trop_{\bbmk}(\mathcal{J}^{\leqslant k}(X; \bbmk))\big)^c 
\end{equation}
Moreover, $\mathrm{S}\big(\Trop_{\Z}(\mathcal{J}^{\leqslant k}(X; \Z))\big)$ and $\mathrm{S}\big(\Trop_{\bbmk}(\mathcal{J}^{\leqslant k}(X; \bbmk))\big)$ are both finite unions of rationally defined convex cones over polyhedrons on the sphere $\mathrm{S}(G)$. In particular, they both have dense rational points. 
\end{theorem}
\begin{remark}
When $k=1$, the first inclusion in \cref{main inclusion} is essentially due to Bieri, Groves and Stebel in \cite{BS80,BS81,BG84}. Moreover, they showed that if $G$ is a finitely generated metabelian group, the first inclusion becomes an equality (for $k=1$). For more details, see \cref{subsection comparing BGS}.
On the other hand, the first inclusion in \cref{main inclusion} could be strict, see \cref{example 2}.


\end{remark}

\cref{thm:compare_two_bnsr} is inspired by Suciu's recent work \cite{Suciu21}. In particular, \cref{thm:compare_two_bnsr} recovers \cite[Theorem 1.1]{Suciu21}, which asserts that
$$\Sigma^k(X;\Z)\subseteq  \mathrm{S}\big(\Trop_{\C}(\mathcal{J}^{\leqslant k}(X;\C))\big)^c.$$ See \cref{rem compare Suciu} for more details. 
One can adapt Suciu's proof to show that $$ \Sigma^k(X;\Z)\subseteq  \mathrm{S}\big(\Trop_{\bbmk}(\mathcal{J}^{\leqslant k}(X; \bbmk))\big)^c $$
for any algebraically closed field coefficients $\bbmk$. 
In general the inclusion 
$$ \bigcup_{\mathrm{char}(\bbmk)=p\geqslant 0}\mathrm{S}(\Trop_{\bbmk}(\mathcal{J}^{\leqslant k}(X; \bbmk)) \subseteq \mathrm{S}\big(\Trop_{\Z}(\mathcal{J}^{\leqslant k}(X; \Z))\big)$$
could be strict, see \cref{example 3}. 
Following  directions pointed out by Bieri and Groves in \cite[Section 8.4]{BG84}, we show that 
the missing ingredient is the tropical variety for the $p$-adic valuation over $\Q$ as in \cref{prop three trop union Z} (see \cref{rem tropical over Z} for more explanations).

The proof of \cref{thm:compare_two_bnsr} replies on a series of nice results due to Bieri, Groves, and Strebel \cite{BS80,BS81,BG84}. 
They gave a complete description for the Sigma-invariants of finitely generated modules over finitely generated abelian groups. Applying their results and a key theorem due to Papadima and Suciu \cite[Theorem 10.1]{PapaSuciu10}, we obtain \cref{thm:compare_two_bnsr}. Since Bieri, Groves, and Strebel's results are one of the origins of tropical geometry (see \cite{EKL}), one can translate the invariant they studied into the language of tropical geometry, and this is why the tropical variety shows up in \cref{thm:compare_two_bnsr}.

\subsection{Applications}


It is a question of Serre to characterize finitely presented groups that can serve as the fundamental group of a compact K\"ahler manifold, called the K\"ahler groups. While some obstructions are known mainly due to the Hodge theory, we still do not have a panorama of this class of groups. The readers may refer to the monographs \cite{ABCKT,PyBook} and the survey papers \cite{Arapura,Burger} for this interesting topic.

A relative version of Serre's question would be to describe the intersection of K\"ahler groups with another class of groups. 
To name a few non-trivial known cases, we have the classification of K\"ahler groups within $3$-dimensional manifold groups in \cite{DS09,Kotschick12,BMS12}; within right-angled Artin groups in \cite{DPS}; within one-relator groups in \cite{BiswasMj}; within groups of large deficiency in \cite{Kotschick}; within cubulable groups up to finite index in \cite{DelzantPy}, etc. Under this spirit, we classify K\"ahler groups among a new class of groups, which is a natural generalization of the right-angled Artin groups. We call them the \textit{weighted right-angled Artin groups}. This class of groups comes from the edge-weighted finite simple graphs. 

\begin{definition}[Weighted right-angled Artin groups]\label{def:weighted RAAG}
Let $\Gamma_\ell= (V, E, \ell)$ be an edge-weighted finite simple graph, with vertex set $V$, edge set
$E$  and an edge weight function $\ell \colon E \to \Z_{>0}$.
The weighted right-angled Artin group associated to  $\Gamma_\ell$ is the group $G_{\Gamma_\ell}$ generated by the vertices $a\in V$, with a defining relation
    $$ [a_i,a_j]^{\ell(e)}=1 $$
for each edge $e = \{a_i,a_j\}$ in $E$ (here $[a_i,a_j]=a_i a_j a_i^{-1}a_j^{-1}$). 
If $\ell(e)=1$ for all $e\in E$, then $G_{\Gamma_\ell}$ is the classical right-angled Artin group, denoted by $G_\Gamma$.
\end{definition}

\begin{remark} The weighted right-angled Artin groups are constructed in a way similar to Artin groups. Moreover, 
    the following Coxeter group  $$ \langle a_i\in V | a_i^2=1, (a_ia_j)^{2\ell(e)}=1 \text{ when there is an edge } e=\{ a_i,  a_j\} \rangle$$ is a quotient of the weighted right-angled Artin group $G_{\Gamma_\ell}$. 
\end{remark}

The various properties of the right-angled Artin group have been thoroughly studied by Papadima and Suciu in \cite{PS06,PS09}. Moreover, 
the K\"ahler right-angled Artin group is classified by Dimca, Papadima, and Suciu as follows (the same result is proved by  Py using different methods in \cite[Corollary 4]{Py13}).

\begin{theorem}[\cite{DPS}, Corollary 11.14] \label{thm RAAG}
Let $\Gamma$ be a finite simple graph and let $G_\Gamma$ denote the corresponding right-angled Artin group. Then the following are equivalent.
\begin{enumerate}[label=(\roman*)] 
    \item The group $G_{\Gamma}$ is K\"ahler.
    \item The graph $\Gamma$ is a complete graph on an even number of vertices. 
    \item The group $G_\Gamma$ is a free abelian group of even rank.
\end{enumerate}  
\end{theorem}
We  classify  K\"ahler weighted right-angled Artin group as follows.
\begin{theorem} \label{thm WRAAG}
For a weighted right-angled Artin group $G_{\Gamma_\ell}$, the following are equivalent.
\begin{enumerate}[label=(\roman*)] 
    \item The group $G_{\Gamma_\ell}$ is K\"ahler.
    \item The edge weighted graph $\Gamma_\ell$ is a complete graph on an even number of vertices and no edges with weight $\geqslant 2$ are adjacent. 
    \item The group $G_{\Gamma_\ell}$ is a finite product of groups with type $\langle a_1,a_2 | [a_1,a_2]^m=1 \rangle$ for some positive integer $m$.
\end{enumerate}   
\end{theorem}

\begin{remark}\label{rmk: delzant's cubulable}
Professor Delzant kindly point out to us that the weighted right-angled Artin group $G_{\Gamma_\ell}$ is cubulable if  the weights $\ell(e)\geqslant 2$ for all edges $e\in E$. In this case, our result is compatible (up to finite index) with his work with Py in \cite{DelzantPy}.
\end{remark}

Dimca, Papadima and Suciu indeed classified quasi-K\"ahler right-angled Artin group in \cite[Theorem 11.7]{DPS}, which leads to the following question.
\begin{question} Can one classify the quasi-K\"ahler weighted right-angled Artin group?
\end{question}

In general, for $G$ a  K\"ahler group, Delzant gave a complete description of $\Sigma^1(G;\Z)$ in \cite{Delzant10}, and Suciu further reinterpreted Delzant's results using the tropical variety of homology jump loci in  \cite[Theorem 12.2]{Suciu21} (see also \cite[Theorem 16.4]{PapaSuciu10}). 
As a continuation of these results, we prove that the first inclusion in \cref{main inclusion} holds as equality for K\"ahler groups in degree $1$. 
Then we derived that the BNS invariant of a K\"ahler group is the same as that of its maximal metabelianization, i.e.
 $$\Sigma^1(G; \Z)=\Sigma^1(G/G''; \Z), $$
 where $G'=[G,G]$, and $G''=[G',G']$. For more details, see \cref{cor:sigma_metabelian_quotient_K\"ahler}. 
 This certainly puts some restrictions on the K\"ahler groups. Furthermore, we summarize some properties for the K\"ahler group in the next proposition. 
Most properties listed here should be already known to the experts. For example, $(viii)\iff (ix)$ follows from Papadima and Suciu's work \cite[Theorem 3.6]{PapaSuciu10}; $(iv)\Rightarrow (vii)$ is proved by Beauville in \cite{Beauville} (see also \cite[Lemme 3.1]{Delzant10} or \cite[Corollary 3.6]{Burger}). (viii) is also related to Arapura's work  \cite[Property $(\mathrm{K}^{-})$]{Arapura}. 

\begin{proposition}\label{prop K\"ahler}
Let $G$ be a K\"ahler group. 
Then the following are equivalent.
\begin{enumerate}[label=(\roman*)] 
    \item $\Sigma^1(G; \Z)=\mathrm{S}(G)$. 
    \item $G'$ is finitely generated.
    \item  $\Sigma^1(G/G''; \Z)=\mathrm{S}(G/G'')$.
    \item $G'/G''$ is finitely generated.
    \item $G/G''$ is polycyclic.
    \item $G/G''$ is finitely presented.
    \item $G/G''$ is virtually nilpotent.
    \item $G'/G''\otimes_{\Z} \bbmk$ is of finite dimensional $\bbmk$- for any field coefficients $\bbmk$.
    \item $\mathcal{V}^1(G; \bbmk)$ consists of only finitely many points  for any algebraically closed field coefficients $\bbmk$.
\end{enumerate}    
\end{proposition}


In addition to investigating K\"ahler groups, we apply \cref{thm:compare_two_bnsr} to the Dwyer-Fried set. In \cite{DwyerFried}, Dwyer and Fried studied when a regular free abelian covering of a finite CW complex admits finite Betti numbers. Their findings were further developed in  \cite{PapaSuciu10,Suciu14abelian_cover,SuciuYZ} with field coefficients. By employing the tropical variety over $\Z$, we extend some of these results to the setting of integral coefficients.

\subsection{Organization}
This paper is organized as follows. In \cref{sec:background}, we recall  Bieri, Groves and Strebel's work. 
In \cref{sec:tropical}, we translate their results into the language of tropical geometry. 
In \cref{sec: proof}, we recall the definitions and properties of the BNSR invariants and jump ideal and give the proof of \cref{thm:compare_two_bnsr}.  In \cref{section examples}, we compute some examples and study the Dwyer-Fried set with $\Z$-coefficients. The last \cref{sec:applications} is devoted to applications on K\"ahler groups. We prove \cref{prop K\"ahler} in \cref{subsec_K\"ahler} and \cref{thm WRAAG} in \cref{subsec WRAAG}.

\section{Bieri, Groves and Strebel's results} \label{sec:background}
In this section, we always assume that $H$ is a finitely generated \textit{abelian} group with $\mathrm{rank}_{\Z} H=n\geqslant 1$. Then $\mathrm{Hom}(H ;\R)\cong \R^n$, and the character sphere 
$$\mathrm{S}(H)=(\mathrm{Hom}(H ;\R)-\{0\})\slash \R^+$$
is topologically an $(n-1)$-dimensional sphere. Here $\R^+$, the set of positive real numbers, acts on $\mathrm{Hom}(H ;\R)-\{0\}$ by scalar multiplication. 
We will abuse the notation $\chi$ for both a nonzero character and its equivalent class $[\chi]$ in $\mathrm{S}(H)$. For any subset $Z\subseteq \mathrm{Hom}(H ;\R)$, denote the image of $Z-\{0\}$ in $\mathrm{S}(H)$ by $\mathrm{S}(Z)$.

Let $R$ be a commutative Noetherian ring with unity. Then the group ring $RH$  is also commutative and Noetherian. Given any nonzero $\chi\in \mathrm{Hom}(H ;\R)$, denote 
$$H_{\chi}=\{h\in H\mid \chi(h)\geqslant 0\}$$
the associated submonoid. Then $RH_\chi$ is a subring of $RH$, hence any $RH$-module can be  viewed as a $RH_\chi$-module.

Following Bieri, Groves and Strebel, for a finitely generated $RH$-module $M$, one can attach the Sigma-invariant $\Sigma(M) \subseteq \mathrm{S}(H)$ defined as
$$\Sigma(M)\coloneqq \{ \chi \in \mathrm{S}(H)\mid M \text{ is finitely generated over }RH_{\chi}\}$$
and $\Sigma^c(M)$ as its complementary in $\mathrm{S}(H)$. 
The set $\Sigma(M)$ plays an important role in answering many algebraic questions, see \cite{BS80,BS81,BG84} for more details. 

Set 
\begin{equation} \label{eq multiply subset}
    \mathscr{S}_\chi \coloneqq \{1+\sum_{h\in H} a_h \cdot h \mid  \text{it is a finite sum with } a_h\in R \text{ and } \chi(h)> 0  \},
\end{equation}
which is a multiplicative subset of $RH$.
Bieri and Strebel gave a complete description of $\Sigma^c(M)$ as follows.

\begin{theorem}[\cite{BS80}, Proposition 2.1]\label{thm BS1}  Let $R$ be a commutative Noetherian ring with unity and $H$ a finitely generated abelian group with $\rank_{\Z} H\geqslant 1$. Assume that $M$ is a  finitely generated $R H$-module with its annihilator ideal denoted by $\mathrm{Ann}(M)$. Then we have  $$ \Sigma^c(M)=\{\chi \in \mathrm{S}(H) \mid \Ann (M) \cap \mathscr{S}_{\chi} =\emptyset\}.$$
\end{theorem}
\begin{remark}
   Bieri and Strebel  proved the above theorem for $R=\Z$, but the given proof remains valid if $\Z$ is generalized to $R$, see \cite[section 1.2]{BS81}. The precise statement as in the above theorem also appeared in the proof of \cite[Theorem 8.1]{BG84}.
\end{remark}
As an application, Bieri and Strebel gave the following computational results.
\begin{theorem}[\cite{BS81}, Theorem 1.1]\label{BS 2}
With the same notations and assumptions as in \cref{thm BS1}, we further assume $\sqrt{\Ann (M)}=\bigcap\limits_{j=1}^q \IP_j$, where $\{\IP_j\}_{j=1}^q$ are all minimal prime ideals containing ${\Ann(M)}$.
Then we have $$\Sigma^c(M) = \Sigma^c(RH/\Ann(M))=\bigcup_{j=1}^q \Sigma^c(RH/\IP _j)  .$$  
In particular, $\Sigma^c(M)$ only depends on the radical ideal $\sqrt{\Ann (M)}$. 
\end{theorem}

Now the computation of $\Sigma^c(M) $ is reduced to the case when $M=R H/ I$ with $I \subsetneq RH $ a proper ideal. 
Bieri, Groves and Strebel further reinterpreted $\Sigma^c(R H/ I) $ by valuations. To explain their results, we recall the definition of valuations on rings.
\begin{definition}[\cite{Bourbaki1998commutative}, Chapter 4]\label{def valuation}
    For a commutative ring $A$ with unity,  a ring valuation $v$ on $A$ is a map $v\colon A\to\R_{\infty}\coloneqq\R\cup \{\infty\}$ such that for any $a,b \in A$ we have that
\begin{enumerate}[label=(\roman*)] 
\item $v(ab)=v(a)+v(b)$,
\item $v(a+b)\geqslant \min\{v(a), v(b)\}$,
\item $v(0)=\infty$ and $v(1)=0$.
\end{enumerate} There may be nonzero elements in $v^{-1}(\infty)$ and it is easy to see that $v^{-1}(\infty)$ is a prime ideal of $A$. When $A$ is a field, this is the classical definition of the valuation on a field.
\end{definition}

We summarize Bieri, Groves and Strebel's results \cite[Theorem 2.1]{BS81} and \cite[Theorem 8.1]{BG84} as follows.

\begin{theorem}\label{thm:BG_main_theorem}
Let $R$ be a commutative Noetherian ring with unity and $H$ a finitely generated abelian group with $\rank_{\Z} H\geqslant 1$.  For a valuation $v$ on $R$ and an ideal $I\subsetneq RH $,  let $\Delta^v_I(H)$ denote the set of all real characters of $H$ induced by valuations on $RH/I$ extending $v$, i.e.
\begin{align}\label{eq: BG Delta set}
\Delta^v_I(H) = \{\chi \in \mathrm{Hom}(H ;\R) \mid & \text{ there exists a valuation } w \colon RH/I \to \mathbb{R}_{\infty} \\
& \text{ such that } (w\circ\kappa)|_{R} = v \text{ and } (w\circ\kappa)|_H = \chi \},\notag
\end{align}
where $\kappa$ is the quotient map $R H\to RH/I$.
Then we have
\begin{equation}\label{eq:BG}
\Sigma^c(RH/I)=\bigcup\limits_{v(R)\geqslant 0} \mathrm{S}(\Delta_I^v(H)),
\end{equation}
where $v$ runs through all valuations of $R$ such that $v(R)\geqslant 0$ (we call this a non-negative valuation for short).   
\end{theorem}

\begin{remark}
Bieri and Groves described  $\Sigma^c(M)$ without  assuming that $R$ is Noetherian, see \cite[Theorem 8.1]{BG84} for more details.  
\end{remark}

Since we mainly focus later on the cases when $R$ is a field or $R=\Z$,  the above theorem in these two cases is explained in detail as follows. 
\begin{example}\label{example:nonnegative valuations}
\begin{enumerate}[label=(\alph*)]
\item Let $R=\bbmk$ be a field.  For any $a\in \bbmk^* \coloneqq \bbmk-\{0\}$, we have $$0=v(1)=v(a)+v(a^{-1}).$$ If $v$ is a non-negative valuation, $v(a)\geqslant 0$ and $v(a^{-1})\geqslant 0$. Hence the non-negative valuation $v$ can only be \textit{the trivial valuation} $v_0$, i.e. 
\begin{center}
    $v_0(a)=0$ for any $a\in \bbmk^*$ and $v_0(0)=\infty$.
\end{center}
 Then for any ideal $I\subseteq \bbmk H$, we have 
$$\Sigma^c(\bbmk H/I)=\mathrm{S}(\Delta^{v_0}_{I}(H)).$$

\item Let $R=\Z$. Then $v^{-1}(\infty)$ is a prime ideal of $\Z$. All valuations on $\Z$ are the following: 
\begin{itemize}
\item If $v^{-1}(\infty)= (p)$ for $p\neq 0$ a prime integer, then for any $a\notin (p)$ and $b\in (p)$, we have $v(a)<\infty=v(b)$, hence $v(a+b)=\min\{v(a),v(b)\}=v(a).$ Thus the valuation $v$ factors through $\Z/p\Z$, which is reduced to a valuation on the residue field $\mathbb{F}_p$. Since for any nonzero element $x\in \mathbb{F}_p$, $x^{p-1}=1$, we have $0=v(1)=v(x^{p-1})=(p-1)v(x)$, which means $v(x)=0$ for any $x$ nonzero.
We denote this valuation on $\Z$ as $\hat{v}_p$, and call it \textit{the mod $p$ valuation}:
    $$\hat{v}_p(n) =
\begin{cases}
  \infty, & \text{if } p \mid n \\
  0, & \text{if } p \nmid n.
\end{cases}$$

\item If $v^{-1}(\infty)=(0)$, $v$ extends to a valuation on $\Q$ defined as $v(\frac{a}{b})=v(a)-v(b)$. By Ostrowski's theorem, it is equivalent to either the archimedean valuation, a $p$-adic non-archimedean valuation $v_p$ or a trivial valuation $v_0$. Noticing that the condition (ii) in Definition \ref{def valuation} is non-archimedean, $v$ has to be \textit{the $p$-adic valuation} $v_p$ or the trivial valuation $v_0$.
\end{itemize}

So it is direct to see that all valuations on $\Z$ are nonnegative.  Then for any ideal $I\subsetneq \Z H$, we have
$$\Sigma^c(\Z H/I)=\mathrm{S}(\Delta^{v_0}_I(H))\cup\bigcup_{p \text{\ prime}}\mathrm{S}\big(\Delta^{v_p}_I(H)\cup  \Delta^{\hat{v}_{p}}_I(H)\big), $$
where the set of primes $p$ in the union is finite thanks to  \cite[Theorem 8.2]{BG84}.

\item  In fact, as long as  $R$ is a discrete valuation domain, there are at most three types of non-negative valuations on $R$ up to multiplication by a positive real number. For more details, see \cite[section 8.4]{BG84}.
\end{enumerate}
\end{example}

The following two important theorems due to Bieri, Groves and Strebel are recorded here for later use.  

\begin{theorem}[\cite{BG84}, Corollarie 8.3 \& 8.4] \label{thm BG dense}
Let $H$ be a finitely generated abelian group with $\mathrm{rank}_{\Z} H \geqslant 1$ and 
$R$ a Dedekind domain. For  a finitely generated $R H$-module $M$,  $\Sigma^c(M)$ is a finite union of rationally defined convex cones over polyhedrons on $\mathrm{S}(H)$. In particular, $\Sigma^c(M)$ has dense rational points. 
\end{theorem}

\begin{theorem}[\cite{BS80}, Theorem 2.4]\label{thm:BS_finitely_generated_module}
Let $H$ be a finitely generated abelian group with $\mathrm{rank}_{\Z} H \geqslant 1$ and $M$ a finitely generated $\Z H$-module. Then the abelian group underlying $M$ is finitely generated over $\Z$ if and only if $\Sigma^c(M)=\emptyset$.
\end{theorem}

\section{Connections with tropical geometry}\label{sec:tropical}
In this section, we focus on the cases where $R$ is a field $\bbmk$ or $R=\Z$.
We will use terminologies in tropical geometry to re-explain  \cref{thm BS1} and \cref{thm:BG_main_theorem}. The readers may refer to the monograph \cite{MacStu} for the required background on tropical geometry.

\subsection{Tropical variety over a valued field}
Let $\bbmk$ be a fixed field endowed with a possibly trivial valuation $v:\bbmk\to \R_{\infty}$. 
We first assume that  $H$ is the \textit{free} abelian group $\Z^n$. The essential modification needed to drop the condition of freeness will be provided in \cref{def torsion} later. Let $\bbmk H=\bbmk[x_1^{{\pm 1} },\ldots,x_n^{{\pm 1} }]$ be the Laurent polynomial ring.

\begin{definition}[Tropical variety over a valued field]\label{def tropical variety}
For an ideal $I\subseteq \bbmk[x_1^{{\pm 1} },\ldots,x_n^{{\pm 1} }]$, there are three  ways to define  \textit{the tropical variety of $I$}. 

\begin{enumerate}[label=(\roman*)] 
    \item For any nonzero $f=\sum\limits_{\textbf{u}\in\Z^n}a_{\textbf{u}} x^{\textbf{u}}\in \bbmk[x_1^{{\pm 1} },\ldots,x_n^{{\pm 1} }]$ and the valuation $v$ on $\bbmk$, 
 the tropical polynomial $\trop_{\bbmk,v}(f)\colon \R^n\to \R$ is defined   by
\begin{equation}\label{eq:trop_of_polynomial}
\trop_{\bbmk,v}(f)({\bf{w}})= \min\limits_{{\textbf{u}}\in \Z^n} \{ v(a_{{\bf{u}}})+{\bf{u}}\cdot{\textbf{w}} \mid a_{\textbf{u}}\neq 0\}, 
\end{equation}
which is a piecewise linear  concave function.
    The tropical hypersurface  associated to $f$ is defined as the set 
    \begin{center}
      $\Trop_{\bbmk,v}(f) \coloneqq\{ \bw\in \R^n \mid \text{ the minimal in\ }  \cref{eq:trop_of_polynomial} \text{ is achieved at least twice}\}$.
    \end{center} 
    
The tropical variety of $I \subseteq  \bbmk[x_1^{{\pm 1} },\ldots,x_n^{{\pm 1} }]$ is defined as 
\begin{equation}\label{defn:trop_first_def}
\Trop_{\bbmk,v}(I)=\bigcap_{f\in I} \Trop_{\bbmk,v}(f). 
\end{equation}

\item 
Fix  $\bw\in \R^n$. 
For any nonzero $f=\sum\limits_{\textbf{u}\in\Z^n}a_{\textbf{u}} x^{\textbf{u}}\in \bbmk[x_1^{{\pm 1} },\ldots,x_n^{{\pm 1} }]$,  the initial form $\mathrm{in}_{\textbf{w},v}(f)$ is the sum of all terms in $f$  where the minimal in  \cref{eq:trop_of_polynomial} is achieved. The initial ideal $\mathrm{in}_{\textbf{w},v}(I)$ is the ideal generated by $\mathrm{in}_{\textbf{w},v}(f)$ where $f$ runs over $I$. Set $$\Trop_{\bbmk, v}(I)\coloneqq \{\textbf{w}\in \R^n\mid 
\mathrm{in}_{\textbf{w},v}(I) \neq \bbmk[x_1^{{\pm 1}},\cdots,x_n^{{\pm 1}}]\}.$$

\item Let $\overline{\bbmk}$ be an algebraically closed field extending $\bbmk$ such that the extension of $v$ to $\overline{\bbmk}$ is nontrivial, still denoted as $v$. Such field always exists. In fact, if the valuation $v$ on $\bbmk$ is nontrivial, one can take $\overline{\bbmk}$ to be the algebraic closure of $\bbmk$. On the other hand, if the valuation $v$ on $\bbmk$ is trivial, one can take $\overline{\bbmk}$ to be the field of Puiseux series $\bigcup_{n\geqslant 1} \K((t^{1/n}))$ if $ \mathrm{char}(\bbmk)=0$ and $\K((t^{\Q})) $ if $\mathrm{char}(\bbmk)>0$. Here $\K$ is an algebraic closure of $\bbmk$. Both fields $\bigcup_{n\geqslant 1} \K((t^{1/n}))$ and $\K((t^{\Q})) $ have nontrivial $\Q$-valued valuation, see e.g. \cite[Example 1.2.2]{EKL}.
The specific choice of $\overline{\bbmk}$ is not important, as long as it is algebraically closed with a nontrivial valuation (see \cite[Theorem 3.2.4 and Remark 3.2.5]{MacStu}). The tropical variety of $I$ is then defined as the closure (under Euclidean topology) of the subset of points $(v(x_1),\cdots,v(x_n))$ where $(x_1,\cdots, x_n)$ belongs to the variety of the ideal $I\otimes_{\bbmk}\overline{\bbmk}$ in $(\overline{\bbmk}^*)^n$.

\end{enumerate}
\end{definition}

The following fundamental theorem of tropical algebraic geometry in \cite[Theorem 3.2.3]{MacStu} shows the equivalence of the above three definitions.

\begin{theorem} \label{thm fundamental trop} (\textbf{The Fundamental theorem of tropical algebraic geometry}) Let $I$ be an ideal in $\bbmk[x_1^{{\pm }},\cdots,x_n^{{\pm }}]$ with a possible trivial valuation $v$ on $\bbmk$. Let $Z=\mathrm{Spec} (\bbmk[x_1^{{\pm } },\ldots,x_n^{{\pm } }]/I)$ denote the corresponding subscheme. Then $Z(\overline{\bbmk})$, the set of $\overline{\bbmk}$-points of $Z$,  is a subvariety in $(\overline{\bbmk}^*)^n$.
With the above notations and assumptions,  the following three subsets of $\R^n$ coincide:
\begin{enumerate}
\item[(i)] the subset $\Trop_{\bbmk,v}(I)$ as defined in \cref{defn:trop_first_def},
\item[(ii)] the set  $\{\textbf{w}\in \R^n\mid \mathrm{in}_{\textbf{w},v}(I)\neq \bbmk[x_1^{{\pm }},\cdots,x_n^{{\pm }}]\}$,
\item[(iii)] the Euclidean closure of the following set of componentwise valuations of points in $Z(\overline{\bbmk})$:
$$v(Z(\overline{\bbmk}))= \{(v(x_1),\cdots,v(x_n))\in \R^n\mid (x_1,\cdots,x_n)\in Z(\overline{\bbmk})\}.$$
\end{enumerate}
In particular, $\Trop_{\bbmk,v}(I)$ only depends on $\sqrt{I}$. If $\sqrt{I}=\bigcap\limits_{j=1}^q \IP_j$, where $\{\IP_j\}_{j=1}^q$ are all minimal prime ideals containing $I$, then
$$ \Trop_{\bbmk,v}(I)=\Trop_{\bbmk,v}(\sqrt{I})=\bigcup_{j=1}^q \Trop_{\bbmk,v}(\IP_j).$$
\end{theorem}
We recall the structure theorem for tropical varieties in \cite[Theorem 3.3.5]{MacStu} with notations and terms explained there in detail.

\begin{theorem}\label{thm structure trop} \textbf{(Structure theorem for tropical variety)} Let $I$ be a prime ideal in $\bbmk[x_1^{{\pm 1}},\cdots,x_n^{{\pm 1}}]$  with $\dim Z(\overline{\bbmk})=d$. Then $\Trop_{\bbmk,v}(I)$ is the support of
a balanced weighted $v(\bbmk^*)$-valued rational polyhedral complex pure of dimension $d$. 
\end{theorem}

Next we define the tropical variety when $H$ is abelian but not necessarily free.

\begin{definition}\label{def torsion}\textbf{(The modification from free abelian to abelian)}
Assume that $H$ has non-trivial torsion part and $H\cong \Z^n\oplus \Z/d_1\Z \oplus \cdots \oplus \Z/d_m\Z$. Then there exists a natural abelian group epimorphism $\psi\colon\Z^{n+m}\twoheadrightarrow H$, which induces a ring epimorphism $\bbmk \Z^{n+m} \twoheadrightarrow \bbmk H$
and an embedding $$\psi^*\colon \mathrm{Hom}(H;\R) \hookrightarrow \mathrm{Hom}(\Z^{n+m};\R). $$
Identify $\bbmk \Z^{n+m} $ with  $\bbmk[x_1^{\pm 1},\cdots,x_n^{\pm 1};y_1^{\pm 1},\cdots,y_m^{\pm 1}]$. Consider $K=( y_1^{d_1}-1,\cdots,y_m^{d_m}-1)$ an ideal in $\bbmk \Z^{n+m}$. Then $\bbmk H \cong \bbmk \Z^{n+m}/K$ is a quotient ring. 
For any ideal $I\subseteq \bbmk H$,   there exists a unique ideal $\tilde{I}\subseteq\bbmk \Z^{n+m}$ containing $K$ such that $\bbmk \Z^{n+m}/\tilde{I}\cong \bbmk H/I.$
For any valuation $v$ on $\bbmk\Z^{n+m}/\tilde{I}$, the relation $y_i^{d_i}-1$ in the ideal $K$ gives $$v(y_i^{d_i})=d_i \cdot v(y_i)=v(1)=0,$$ which implies that $v(y_i)=0$ for any $1\leqslant i \leqslant m$. Since $ \tilde{I}\supseteq K$, $$\Trop_{\bbmk,v}(\tilde{I}) \subseteq \Trop_{\bbmk,v}(K) =\psi^*(\mathrm{Hom}(H;\R)).$$ 
Then we can define $$\Trop_{\bbmk,v}(I) \coloneqq (\psi^*)^{-1}\Trop_{\bbmk,v}(\tilde{I}).$$
\end{definition}

\begin{remark} \label{rem Suciu}
One can also understand $\Trop_{\bbmk,v}(I)$ as follows.  
Consider an algebraically closed field $\overline{\bbmk}$ containing $\bbmk$ with nontrivial valuation as in \cref{def tropical variety}(iii).  
The variety associated with the coordinate ring $\overline{\bbmk} H$ is $\coprod (\overline{\bbmk}^*)^n$, a finite disjoint union of $(\overline{\bbmk}^*)^n$,  
with identity $1$ corresponding to the trivial representation $H\to \overline{\bbmk}^*$. The connected component containing $1$ is an affine torus $(\overline{\bbmk}^*)^n$ and any other connected component is a translation of this one, by characters of finite order induced by the torsion part of $H$. Note that for any character of finite order, its valuation has to be $0$. 
Let $Z(\overline{\bbmk} )$ be the $\overline{\bbmk} $-points of $I$ and denote its connected components as $\{ Z_j(\overline{\bbmk} )\}_{1\leqslant j \leqslant q}$. Each connected component is contained in one of the connected components of $\coprod (\overline{\bbmk}^*)^n$. Up to translation via a torsion element, all the connected components can be considered as a subset in $(\overline{\bbmk}^*)^n.$
Then the tropical variety $\Trop_{\bbmk,v}(I) $ is indeed the set
\begin{equation}
\bigcup\limits_{1\leqslant j \leqslant q} \overline{ \{(v(x_1),\cdots,v(x_n))\mid (x_1,\cdots,x_n)\in Z_j(\overline{\bbmk})\subset  (\overline{\bbmk}^*)^n\}}.
\end{equation}
\end{remark}

Einsiedler, Kapranov, and Lind first proved in \cite[Corollary 2.2.6]{EKL} that the tropical variety is indeed the set  defined by Bieri and Groves using valuations. We slightly generalize their results as follows.  
\begin{proposition} \label{prop EKL}
Consider a valuation $v$ on $\bbmk$ and a finitely generated abelian group $H$ with $\rank_\Z(H)\geqslant 1$. For an ideal $I \subsetneq \bbmk H$, let $ \Delta^v_I(H)$ denote the set considered in  \cref{thm:BG_main_theorem}. Then we have
    $$ \Delta^v_I(H)= \Trop_{\bbmk,v} (I).$$
\end{proposition}
\begin{proof}
If $H$ is free abelian, the claim is proved in \cite[Corollary 2.2.6]{EKL}. Now we assume that $H$ has a non-trivial torsion part and we use the notations as in \cref{def torsion}.
Then the claim holds if we have the following equalities:
$$ \Delta^v_I(H)=(\psi^*)^{-1}\Delta^v_{\tilde{I}}(\Z^{n+m})=  (\psi^*)^{-1}\Trop_{\bbmk,v}(\tilde{I})=\Trop_{\bbmk,v}(I). $$
The second equality follows from the free abelian case, and the last one follows from \cref{def torsion}.
So we are left to prove the first equality. Since $\bbmk \Z^{n+m}/\tilde{I}\cong \bbmk H/I ,$
for any valuation  $w \colon \bbmk H/I \to \mathbb{R}_{\infty} $  with $w|_{\bbmk} = v$ and $w|_H = \chi $, it is also a valuation on $\bbmk \Z^{n+m}/\tilde{I}$ such that $w|_{\bbmk} = v$. Since $\tilde{I}\supseteq K$, it is easy to see that  $w|_{\Z^{n+m}} \in \psi^*(\mathrm{Hom}(H;\R))$. In particular,  $w|_{\Z^{n+m}} =\psi^* \chi$. This gives a one to one correspondence between  $\Delta^v_I(H)$ and $(\psi^*)^{-1}\Delta^v_{\tilde{I}}(\Z^{n+m})$, which implies the first equality. 
 \end{proof}

\subsection{Tropical variety over rings} \label{sec: trop over ring}
In this subsection, we follow \cite[Section 1.6]{MacStu} to define the tropical variety over a commutative Noetherian ring $R$. Now $H$ is a finitely generated abelian group, not necessarily free.

Fix a character $\chi \in \mathrm{Hom}(H ;\R)$. 
For any nonzero $f=\sum a_h h \in R H$, we denote by $\deg_\chi(f)$ the minimal value of $\chi(h)$ with $a_h\neq 0$, and call it the $\chi$-degree of $f$. 
The initial form $\mathrm{in}_\chi(f)$ is the sum of all terms $a_h h$ in $f$ such that $\chi(h)=\deg_\chi(f)$. For an ideal $I \subseteq RH$, its initial ideal is defined as
$$\mathrm{in}_\chi (I)\coloneqq \langle\mathrm{in}_\chi(f)\mid f\in I\rangle.$$  
Following  \cite[Section 1.6]{MacStu}, we define the tropical variety over $R$ below. 

\begin{definition}[Tropical variety over a ring]\label{def tropical ring}
The {\it tropical variety over $R$} of an ideal $I\subseteq RH $ is the following set 
$$\Trop_R(I)\coloneqq\{\chi \in \mathrm{Hom}(H ;\R) \mid \mathrm{in}_\chi (I)\neq R H \}. $$
\end{definition}

Note that for the zero character $\mathrm{in}_{0} I=I$.
Hence $0\in \Trop_R(I)$ if and only if $I$ is a proper ideal in $RH$, which implies that
$\Trop_R (I)=\emptyset$ if and only if $I=RH$. 
From now on we always assume that $I$ is a \textit{proper} ideal. Moreover, by definition if $\chi\in\Trop_R (I)$, then $r\cdot\chi\in\Trop_R (I)$ for any positive real number $r$. Therefore, $\mathrm{S}(\Trop_R(I))$ shares the same information as $\Trop_R(I)$. 
The following result shows that the work of Bieri, Groves and Strebel can be reinterpreted by the tropical variety.  

\begin{proposition}\label{prop trop equals Sigma complement}
With the above assumptions and notations,
we have
$$\mathrm{S}\big(\Trop_R(I)\big) =\Sigma^c(RH/I).$$
Moreover, if $\sqrt{I}=\bigcap\limits_{j=1}^q \IP_j$, where $\{\IP_j\}_{j=1}^q$ are all minimal prime ideals containing $I$,
then 
$$\Trop_R(I)=\bigcup_{j=1}^q \Trop_R(\IP _j).$$  
In particular, $\Trop_R(I)$ only depends on the radical ideal $\sqrt{I}$.
\end{proposition}

\begin{proof}
By \cref{thm BS1}, $\chi \in \Sigma^c(RH/I)$ if and only if $\mathscr{S}_{\chi}\cap I=\emptyset$ with notation in \cref{eq multiply subset}. By \cref{def tropical ring}, it is clear that $\chi\in \mathrm{S}(\Trop_R(I))$ if and only if $\mathrm{in}_\chi(I)\neq RH$.
Thus we only need to show that $\mathscr{S}_{\chi}\cap I\neq \emptyset$ if and only if $\mathrm{in}_\chi(I)= R H$.

One direction is clear. If $f\in \mathscr{S}_{\chi}\cap I$, then $\mathrm{in}_{\chi}(f)=1$ and $\mathrm{in}_{\chi}(I)=R H$. Conversely, if $\mathrm{in}_\chi(I) = R H $, then there exist $f_1,\cdots,f_k\in I$ and $g_1, \cdots, g_k\in R H$ such that 
$$1=\sum_{j=1}^k \mathrm{in}_\chi(f_j) \cdot g_j.$$
Let  $g'_j$ denote the $\chi$-homogeneous terms of $g_j$ with $\chi$-degree being  $-\deg_\chi(f_j).$ 
Then we have $$ 1=\sum_{j=1}^k \mathrm{in}_\chi(f_j) \cdot g'_j$$
Set $f=\sum_{j=1}^k f_j\cdot g'_j$. It is clear that $f\in I$ and $\mathrm{in}_\chi(f)=1 $, hence  $f\in \mathscr{S}_{\chi}\cap I$.  Then the first claim follows. 
The moreover part is a direct consequence of \cref{BS 2}.
\end{proof}

When $R$ is a field $\bbmk$, the tropical variety $\Trop_\bbmk(I)$ defined in \cref{def tropical ring} is indeed  $\Trop_{\bbmk, v_0}(I)$ with the trivial valuation $v_0$ on $\bbmk$ defined in \cref{def tropical variety}. 
The claim follows from \cref{example:nonnegative valuations}(a) and \cref{prop EKL}.

When $R=\Z$, the tropical variety over $\Z$ can be understood using the tropical varieties over various fields with valuations in \cref{example:nonnegative valuations}(b). We summarize this in the next proposition and its proof follows the idea in \cite[section 2.1 \& 2.2]{BG84}

\begin{proposition}\label{prop three trop union Z}
Let $I$ be an ideal in $\Z H$. Then we have
$$\mathrm{S}\big(\Trop_\Z(I)\big) = \mathrm{S}\Big(\Trop_{\Q,v_0}(I\otimes_\Z \Q)\cup\bigcup\limits_{\substack{p \text{ prime} \\ \text{finitely many}}}\big(\mathrm{Trop}_{\Q,v_p}(I\otimes_\Z \Q)\cup \mathrm{Trop}_{\bF_p,\hat{v}_p}(I\otimes_\Z \Fp)\big)
\Big),$$
where in the union we are taking the trivial valuation $v_0$ on $\Q$, the $p$-adic valuation $v_p$ on $\Q$ and  the trivial valuation $\hat{v}_p$ on $\bF_p$, respectively. 
\end{proposition} 
\begin{proof}
By \cref{thm:BG_main_theorem}, \cref{example:nonnegative valuations}(b) and \cref{prop trop equals Sigma complement}, we need to show that if $v$ is one of the valuations $v_p$ with $p\geqslant 0$ or $\hat{v}_p$ with $p>0$, then 
\begin{equation}\label{eq: Delta equals trop}
\Delta^{v}_I(H)=\Trop_{\bbmk,v}(I\otimes_\Z\bbmk)    
\end{equation}
for a proper field $\bbmk$.

We first study the trivial and $p$-adic valuations $v_p$ over $\Q$ with $p\geqslant 0$. 
If $ I\cap \Z\neq (0)$, write $I\cap \Z=(m)$ with $m\neq 0$. Then any valuation $w$ on $\Z H/I$ gives $w(m)=\infty$, while $w|_\Z=v_p$ can not happen since $v_p(m)\neq \infty$. Hence $\Delta^{v_p}_I(H)=\emptyset$. Meanwhile, $m\in I$ implies $ I\otimes_\Z\Q= \Q H$, hence $\mathrm{Trop}_{\Q,v_p}(I\otimes_\Z\Q)=\emptyset$. So \cref{eq: Delta equals trop} holds under this assumption.  
 
Now we assume that  $I\cap \Z=(0)$. In this case, the ring $\Q H/(I\otimes_\Z \Q)$  is non-trivial. Set $S=\Z\backslash\{0\}$, which is a multiplicative subset of $\Z$. Since $v_p^{-1}(\infty)=(0)$ in $ \Z$,  we have $S\cap (0)=\emptyset$. Then there exists a unique valuation, still denoted as $v_p : S^{-1}\Z=\Q\to \R_{\infty}$, given by ${v_p}(\frac{a}{b})=v_p(a)-v_p(b)$, $a\in\Z, b\in S$, i.e. the $p$-adic valuation on $\Q$. Hence any valuation $w$ on $\Z H/I$ with $w|_\Z =v_p$ and $w|_H=\chi $ gives a unique valuation $w'$ on $\Q H/(I\otimes_\Z \Q)$ with $w'|_\Q =v_p$ and $w'|_H=\chi $.  
Therefore, one obtains readily from \cref{prop EKL} that 
$$\Delta_I^{v_p}(H)=\Delta^{{v_p}}_{I\otimes_\Z \Q}(H)=\Trop_{\Q,v_p}(I\otimes_\Z \Q).$$
This is the first reduction step considered by Bieri and Groves in \cite[section 2.2]{BG84}.

Next we study the mod $p$-valuation $\hat{v}_p$ with $p>0$. 
If $I\otimes_\Z \bF_p = \bF_p H $, then there exists $f\in I$ such that $f=1+pf'$. Since $f\in I$, any valuation $w$ on $\Z H/I$ gives $w(f)=\infty$. On the other hand, $\hat{v}_p(p)= \infty$ implies that $w(f)=w(1+pf')=w(1)=0$, a contradiction. Hence  $\Delta^{\hat{v}_p}_I(H)=\emptyset$. Meanwhile, $I\otimes_\Z \bF_p = \bF_p H$ implies  $\mathrm{Trop}_{\bF_p,\hat{v}_p}(I\otimes_\Z\bF_p)=\emptyset$. So \cref{eq: Delta equals trop} holds under this assumption. 

If $I\otimes_\Z \bF_p \neq \bF_p H $, then the ring $\bF_p H/(I\otimes_\Z \bF_p)$ is non-trivial. Fix a valuation $w$ on $\Z H/I$ with $w|_\Z =\hat{v}_p  $ and $w|_H=\chi $. Since $ \hat{v}_p(p)=\infty$, $ w$ can also be viewed as a valuation on $\Z H/\langle p,I \rangle$, where $\langle p,I \rangle$ is the ideal generated by $p$ and $I$. Note that $\Z H/\langle p,I \rangle\cong \bF_p H/(I\otimes_\Z \bF_p)$. Then $w$ can also be viewed as a valuation on $\bF_p H/(I\otimes_\Z \bF_p)$. Clearly $w|_{\bF_p} $ is the trivial valuation and $w|_H=\chi $. From \cref{prop EKL}, we have
$$\Delta_I^{\hat{v}_p}(H)=\Delta^{\hat{v}_p}_{I\otimes_\Z \Fp}(H)=\Trop_{\Fp,\hat{v}_p}(I\otimes_\Z \Fp).$$
This is the third reduction step considered by Bieri and Groves in \cite[section 2.2]{BG84}.
\end{proof}

\begin{remark} \label{rem tropical over Z}
In general, we have
$$ \Trop_\Z(I) \neq \Trop_{\Q,v_0}(I\otimes_{\Z}\Q)\cup\bigcup\limits_{\substack{p \text{ prime} \\ \text{finitely many}}}\big(\mathrm{Trop}_{\Q,v_p}(I\otimes_{\Z}\Q)\cup \mathrm{Trop}_{\bF_p,\hat{v}_p}(I\otimes_{\Z}\Fp)\big),$$ as shown by \cref{example trop comparision} and \cref{fig:total} below. But \cref{prop three trop union Z} shows that they coincide after projection onto the unit sphere. This shows the central role of the $p$-adic tropicalization (if it is distinct from the trivial one). Its asymptotic behavior gives the trivial tropicalization and its local behavior near the origin gives the mod $p$ tropicalization. See \cite[Theorem C1, C2]{BG84} for more details.   
\end{remark}

\begin{example} \label{example trop comparision}
Consider the ideal $I=(x_1+x_2-2) \subseteq \Z[x_1^{\pm 1},x_2^{\pm 1}]$.   See \cref{fig:total} below for its various tropicalizations, where \cref{fig:z-trop} is $\Trop_\Z(I)$ and \cref{fig:trivial-valuation}, \cref{fig:mod2-valuation} and \cref{fig:2adic-valuation} are the tropical varieties of $I$ considered in $ (\Q[x_1^{\pm 1},x_2^{\pm 1}],v_0)$, $(\mathbb{F}_2[x_1^{\pm 1},x_2^{\pm 1}],\hat{v}_2)$ and $(\Q[x_1^{\pm 1},x_2^{\pm 1}], v_2)$, respectively. 
Note that the union of \cref{fig:trivial-valuation,fig:mod2-valuation,fig:2adic-valuation} is different from \cref{fig:z-trop}, while the projections onto the unit sphere are the same (see \cref{fig:all_three}). 
\end{example}

\begin{figure}[ht!]
    \centering

    \begin{subfigure}[b]{0.23\textwidth}
        \centering
        \begin{tikzpicture}[scale=0.4]
            \draw[-,dashed] (-3,0) -- (3,0);  
            \draw[-,dashed] (0,-3) -- (0,3);  
            \draw[-,blue,thick] (0,0) -- (0,3.2);  
            \draw[-,blue,thick] (0,0) -- (-3,-3);  
            \draw[-,blue,thick] (0,0) -- (3.2,0);  
        \end{tikzpicture}
        \caption{The trivial Trop.}
        \label{fig:trivial-valuation}
    \end{subfigure}
    \hfill
    \begin{subfigure}[b]{0.23\textwidth}
        \centering
        \begin{tikzpicture}[scale=0.4]
            \draw[-,dashed] (-3,0) -- (3,0);  
            \draw[-,dashed] (0,-3) -- (0,3);  
            \draw[-,blue,thick] (3,3) -- (-3,-3);  
        \end{tikzpicture}
        \caption{The mod $2$ Trop.}
        \label{fig:mod2-valuation}
    \end{subfigure}
    \hfill
    \begin{subfigure}[b]{0.23\textwidth}
        \centering
        \begin{tikzpicture}[scale=0.4]
            \draw[-,dashed] (-3,0) -- (3,0);  
            \draw[-,dashed] (0,-3) -- (0,3);  
            \draw[-,blue,thick] (1,1) -- (-3,-3);  
            \draw[-,blue,thick] (1,1) -- (1,3);  
            \draw[-,blue,thick] (1,1) -- (3,1);  
        \end{tikzpicture}
        \caption{The $2$-adic Trop.}
        \label{fig:2adic-valuation}
    \end{subfigure}
    \hfill
    \begin{subfigure}[b]{0.23\textwidth}
        \centering
        \begin{tikzpicture}[scale=0.4]
            \draw[-,dashed] (-3.2,0) -- (3,0);  
            \draw[-,dashed] (0,-3.2) -- (0,3);  
            \draw[-,blue,thick] (0,0) -- (0,3.2);  
            \draw[-,blue,thick] (0,0) -- (-3,-3);  
            \draw[-,blue,thick] (0,0) -- (3.2,0);  
            \fill[blue!50] (0,0) -- (3.2,0) -- (3.2,3.2) -- (0,3.2) -- cycle;
        \end{tikzpicture}
        \caption{The $\Z$-Trop.}
        \label{fig:z-trop}
    \end{subfigure}

    \vspace{0.3cm} 

    \begin{subfigure}[b]{0.45\textwidth}
        \centering
        \begin{tikzpicture}[scale=0.4]
            \draw[-,dashed] (-3,0) -- (3,0);
            \draw[-,dashed] (0,-3) -- (0,3);
            \draw (0,0) circle (2);
            \fill[blue] (2,0) circle (4pt);
            \fill[blue] (0,2) circle (4pt);
            \fill[blue] (-1.414213562, -1.414213562) circle (4pt);
            \fill[blue] (1.414213562, 1.414213562) circle (4pt);
        \end{tikzpicture}
        \caption{Projection of \ref{fig:trivial-valuation} and \ref{fig:mod2-valuation}.}
        \label{fig:trivial_and_mod p}
    \end{subfigure}
    \hfill
    \begin{subfigure}[b]{0.45\textwidth}
        \centering
        \begin{tikzpicture}[scale=0.4]
            \draw[-,dashed] (-3,0) -- (3,0);
            \draw[-,dashed] (0,-3) -- (0,3);
            \draw (0,0) circle (2);
            \fill[blue] (2,0) circle (4pt);
            \fill[blue] (0,2) circle (4pt);
            \fill[blue] (-1.414213562, -1.414213562) circle (4pt);
            \draw[blue, thick] 
                (2,0) arc[start angle=0, end angle=90, radius=2];
        \end{tikzpicture}
        \caption{Projection of \ref{fig:z-trop} or the union of \ref{fig:trivial-valuation}, \ref{fig:mod2-valuation}, and \ref{fig:2adic-valuation}.}
        \label{fig:all_three}
    \end{subfigure}
    \caption{Comparison of several tropicalizations.}
    \label{fig:total}
\end{figure}
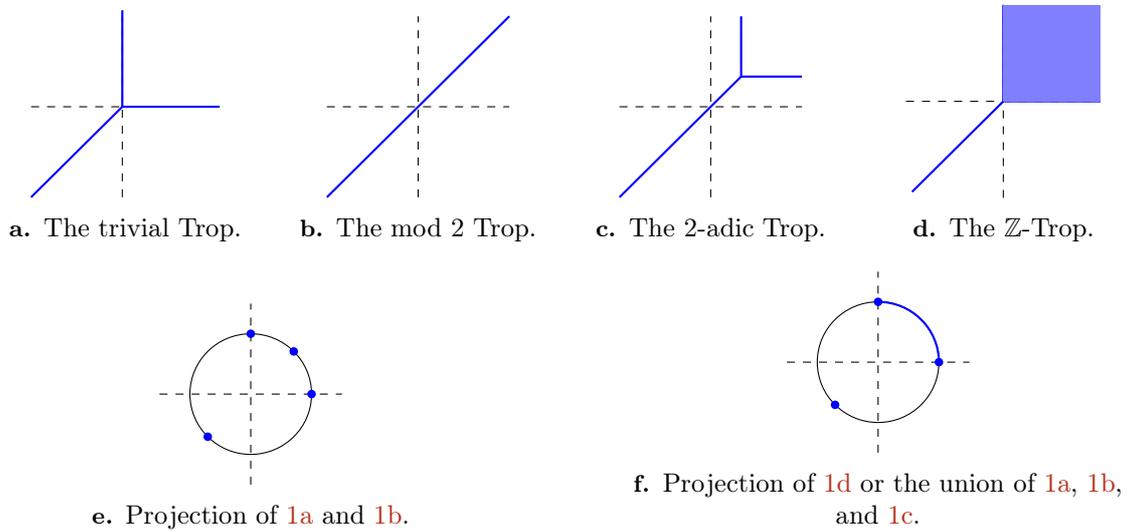

We end this section with a property for tropical varieties.  Let $\psi\colon H \twoheadrightarrow H'$ be an epimorphism of abelian groups with kernel $N$. It induces  an embedding
$$\psi^*\colon \mathrm{Hom}(H'; \R) \hookrightarrow \mathrm{Hom}({H}; \R)$$
and a ring epimorphism $\psi_*\colon 
R{H} \twoheadrightarrow RH'$. It is easy to see that the kernel $K$ of $\psi_*$ is the ideal generated by $\{n-1\mid n\in N\}$ in $R{H}$. 
For an ideal $I'\subsetneq RH'$, 
 we have $ R {H}/{\psi_*^{-1}(I')} \cong RH'/I'$. 




\begin{proposition}\label{prop functor} With the above notations and assumptions, we have 
    $$ \Trop_R({\psi_*^{-1}(I')}) =  \psi^*(\Trop_R(I')).$$
\end{proposition}
\begin{proof}
By \cref{prop trop equals Sigma complement} and \cref{thm:BG_main_theorem}, we have \begin{center}
    $ \Trop_R(I')=\bigcup\limits_{v(R)\geqslant 0}\Delta_{I'}^v(H') $
and  $ \Trop_R(\psi_*^{-1}(I'))=\bigcup\limits_{v(R)\geqslant 0}\Delta_{\psi_*^{-1}(I')}^v(H) $.
\end{center}
Note the two facts  $\Trop_R(K)=\psi^*(\mathrm{Hom}(H'; \R))$ and $K\subseteq \psi_*^{-1}(I')$. 
By the ring isomorphism $ R {H}/{\psi_*^{-1}(I')} \cong RH'/I'$, one readily sees that $ \Trop_R({\psi_*^{-1}(I')}) =  \psi^*(\Trop_R(I')).$
\end{proof}

\section{Proof of Theorem \ref{thm:compare_two_bnsr}}\label{sec: proof}
In this section, we apply the results in \cref{sec:background,sec:tropical} to prove \cref{thm:compare_two_bnsr}.

\subsection{The BNSR invariants revisited}

We start with a finiteness condition for chain complexes, following the approach of Farber, Geoghegan, and Sch\"utts in \cite{FGS}. 

\begin{definition} \label{def:chain_complex}
Let $C = (C_i,\partial_i)_{i\geqslant 0}$ be a non-negatively graded chain complex over a ring $A$. For each integer  $k\geqslant 0$,  we say $C$ is of finite $k$-type if there is a chain complex $C'$ of finitely
generated projective (left) $A$-modules and a chain
map $C' \to C$ inducing isomorphisms $H_i(C') \to H_i(C)$ for $i < k$ and an epimorphism
$H_k(C') \to H_k(C).$
\end{definition}

 
Let $X$ be a connected finite CW-complex with the fundamental group $G$.
Denote $\widetilde{X}$ the universal covering of $X$. The cell structure of $X$ lifts to cell structures on the universal cover $\widetilde{X}$ with $G$-action via deck transformations. 
Thus, the cellular chain complex $C_*(\widetilde{X};\Z)$ is a complex of finitely generated free $\Z G$-modules.
Given any nonzero $\chi\in \mathrm{Hom}(G ;\R)$, the set $G_{\chi}=\{g\in G\mid \chi(g)\geqslant 0\}$ 
is a submonoid of $G$, which depends only on $[\chi]\in \mathrm{S}(G)$. Then $C_*(\widetilde{X};\Z)$ can also be viewed as a complex of $\Z G_\chi$-modules. The following definition of the BNSR invariant of $X$ can be found in \cite{FGS}.

\begin{definition}\label{def BNSR}
    For each integer $k\geqslant 0$, the $k$-th BNSR invariant of $X$ is given by
    $$\Sigma^k(X; \Z) \coloneqq\{\chi\in \mathrm{S}(G)\mid  C_*(\widetilde{X};\Z) \text{ is of finite $k$-type over } \Z G_\chi \}. $$
\end{definition}


We denote by $\Sigma^k(X; \Z)^c $ the complement of $\Sigma^k(X; \Z)$ in $\mathrm{S}(G)$. It is shown in \cite{FGS} that $\Sigma^k(X; \Z)$ is an open subset of $\mathrm{S}(G)$  and depends only on the homotopy type of $X$. 
In particular, $\Sigma^1(X; \Z)$ depends only on  $G$, hence one can also denote it by $\Sigma^1(G; \Z)$. This is (almost) the BNS invariant of $G$, which can be defined via the Cayley graph as follows. 
One picks a finite generating set of $G$ and let $\Gamma(G)$ be the
corresponding Cayley graph of $G$. 
For any $\chi \in \mathrm{S}(G)$, let $\Gamma_\chi(G)$ be the full subgraph on the vertex set $G_\chi$. 
\begin{definition}[\cite{BNS}]
Let $G$ be a finitely generated group. The BNS invariant $\Sigma^1(G)$ consists of  $\chi \in \mathrm{S}(G)$ 
for which the graph $\Gamma_\chi(G)$ is connected.
\end{definition}

As noted by Bieri and Renz in \cite[Section 1.3]{BR88},  $\Sigma^1(G)=-\Sigma^1(G; \Z)$. 
In particular, $\Sigma^1(G)$ does not depend on the choice of finite generating set for $G$. To make the notations consistent, we always use $\Sigma^1(G;\Z)$ instead of $\Sigma^1(G)$, and all the conclusions originally about $\Sigma^1(G)$ will be rewritten with respect to $\Sigma^1(G;\Z)$. 
 The complements of the BNS invariants enjoy the following naturality property. 
 \begin{proposition}[Proposition 3.3, \cite{BNS}] \label{prop naturality}
 Let $\psi\colon G\twoheadrightarrow Q$ be a surjective group homomorphism between finitely generated groups. Then the induced embedding $\psi^* \colon S(Q) \hookrightarrow \mathrm{S}(G)$, restricts to an injective map $ \psi^* \colon \Sigma^1(Q;\Z)^c \hookrightarrow \Sigma^1(G;\Z)^c$.
 \end{proposition}
Bieri, Neumann, and Strebel showed that the BNS invariants are important in controlling the finiteness properties of kernels of abelian quotients. 


\begin{theorem}[Theorem B1,\cite{BNS}]\label{thm BNS}
Let $G$ be a finitely generated group and let $N$ be a normal subgroup of $G$ with an abelian quotient. Denote
$${\mathrm{S}(G,N)=\{\chi\in \mathrm{S}(G)\mid \chi(N)=0\}}.$$ Then $N$ is finitely generated if and only if $S(G, N)\subseteq -\Sigma^1(G;\Z)$. In particular, $G'$ is finitely generated if and only if $\mathrm{S}(G)=\Sigma^1(G;\Z)$.
 \end{theorem}

\subsection{Jump loci ideal and Alexander ideal}  Let $H=H_1(X; \Z)$, the abelianization of $G$. 
Denote $X^H$ the maximal abelian covering of $X$. 
Let $R$ be $\Z$ or a field $\bbmk$.
Similarly. the cellular chain complex $X^H$ with $R$-coefficients, $C_{*}(X^H; R)$, is a bounded complex of finitely generated free $R H$-modules:

\begin{equation}
    \label{chain compelx}
 \cdots  \to  C_{i+1}(X^H; R) \xrightarrow{\partial_{i}} C_i(X^H; R) \xrightarrow{\partial_{i-1}} C_{i-1}(X^H; R) \to \cdots \xrightarrow{\partial_0} C_0(X^H; R)  \to 0 .
\end{equation} 

\begin{definition}\label{def Alexander}
The {\it $i$-th Alexander invariant} $ H_{i}(X^H; R)$  is  the 
$i$-th homology of  $C_*(X^H; R)$ as chain complex of $R H$-modules.
The \textit{$i$-th Alexander ideal $\mathrm{Ann}( H_i(X^H; R))$} is the annihilator ideal of $H_{i}(X^H; R)$ as finitely generated $RH$-modules.
\end{definition} 

\begin{definition}\label{defn:jump ideal}
    The \textit{$i$-th jump ideal of $X$} is defined as
$$\mathcal{J}^i(X; R)=I_{c_i}(\partial_{i}\oplus\partial_{i-1})$$
where $c_i=\mathrm{rank}(C_i(X^H; R))$ as free $RH$-modules and  $I_{c_i}(-)$ denotes the ideal generated by size $c_i\times c_i$ minors of the matrix and is commonly referred to as the Fitting ideal.
\end{definition}
When $R=\bbmk$, the maximal spectrum of these two types of ideals are under the more well-known names: Alexander varieties and homology jump loci, see e.g. \cite{PapaSuciu10, PS14}. 
Moreover, Papadima and Suciu established a comparison between
these two types of ideals.

\begin{theorem}[\cite{PapaSuciu10} Theorem 3.6, or \cite{PS14} Theorem 2.5] \label{thm PS two types of ideals} 
With the above notations, assumptions and $R=\Z$ or a field $\bbmk $, for any $k\geqslant 0$ we have $$\sqrt{\mathcal{J}^{\leqslant k}(X; R)}=\sqrt{\mathrm{Ann} (H_{\leqslant k}(X^H; R))},$$
where $\mathcal{J}^{\leqslant k}(X; R)=\bigcap\limits_{0\leqslant j\leqslant k}\mathcal{J}^j(X; R)$ and $\mathrm{Ann}(H_{\leqslant k}(X^H; R))=\bigcap\limits_{0\leqslant j\leqslant k}\mathrm{Ann}(H_{j}(X^H; R))$. 
\end{theorem}

\begin{remark}
Papadima and Suciu proved the above results for $R=\bbmk$, while their argument goes without any difficulty to $R=\Z$, see the proof of \cite[Theorem 2.5]{PS14}.
\end{remark}

In homological degree one, both $\cJ^1(X; R)$ and $\Ann(H_1(X^H; R))$ depend only on the fundamental group $G$ (in fact only on $G/G''$, see  e.g. \cite[Section 2.5]{Suciu14abelian_cover}). One can thus denote $\cJ^1(X; R)$  by $\cJ^1(G; R)$.  
 Moreover, one may compute them by abelianizing the matrix of Fox derivatives of $G$, see e.g. \cite{Fox}.

If $R=\bbmk$ is an algebraically closed field, the group of $\bbmk$-valued  characters, $ \mathrm{Hom}(G; \bbmk^*)$, is a commutative affine algebraic group. Each character $\rho$ in $\mathrm{Hom}(G; \bbmk^*)$ defines a rank one local system on $X$, denoted by $L_{\rho}$. Since $\bbmk^*$ is abelian, we have $\mathrm{Hom}(G; \bbmk^*)=\mathrm{Hom}(H;\bbmk^*)$.

\begin{definition} \label{def homology jump loci}
   The $i$-th homology jump loci of $X$ (over $\bbmk$) are defined as
$$\sV^i(X; \bbmk)\coloneqq \lbrace \rho\in \mathrm{Hom}(G; \bbmk^*) \mid  H_{i}(X; L_{\rho})\neq 0 \rbrace.$$ 
\end{definition}

By definition $\sV^i(X; \bbmk)$ is the variety  of the ideal $\cJ^i(X; \bbmk)$. Hence $\sV^1(X;\bbmk)$ depends only on the fundamental group $G$ and one can denote it by $\sV^1(G;\bbmk)$.

\bigskip
 
In his thesis, Sikorav reinterpreted the BNS invariant of a finitely generated group in terms of the Novikov homology in degree $1$ (see \cite{Sikorav} or \cite{Sikorav_Survery}). It was later generalized to higher degrees by Bieri  and Renz in \cite{BR88} and from groups to spaces by Farber, Geoghegan and Sch\"utts in \cite{FGS}.  As an application of this interpretation, Papadima and Suciu proved the following theorem, which connects the BNSR invariants with the homology jump loci of $X$. The degree $1$ case originates in the work of Delzant \cite[Proposition 1]{Delzant08}.

\begin{theorem} [Theorem 10.1, \cite{PapaSuciu10}] \label{thm PS} 
Let $X$ be a connected finite CW complex and let $\bbmk$ be a field. Suppose $ \rho\colon G\to \bbmk^*$ is a multiplicative character such that $H_i(X; L_\rho)\neq 0$ for some $0\leqslant i \leqslant k$. Let $v \colon \bbmk^*\to \R$ be a valuation on $\bbmk$ such that $\chi=v \circ \rho$ is non-zero. Then we have $\chi \notin \Sigma^k(X; \Z)$. 
\end{theorem}


\subsection{Proof of Theorem \ref{thm:compare_two_bnsr}}\label{section proof of theorem 1.1}
For a field $\bbmk$, note that $\mathcal{J}^{\leqslant k}(X; \bbmk) = \mathcal{J}^{\leqslant k}(X; \Z)\otimes \bbmk$.
Then the second inclusion in \cref{main inclusion} follows directly from \cref{def tropical ring}. So
we only need to prove the first inclusion in \cref{main inclusion}.

We assume that $\cJ^{\leqslant k}(X; \Z)$ is a proper ideal in $\Z H$, otherwise 
the claim is vacuous. Let $\IP$ be a minimal prime ideal containing $ \sqrt{\cJ^{\leqslant k}(X; \Z)}$.
Fix a valuation $v$ on $\Z$ as studied in Example \ref{example:nonnegative valuations}(b), which is automatically non-negative.
Consider a nonzero $\chi \in \Delta^v_\IP(H)$ with notations as in \cref{thm:BG_main_theorem}.
    Then there exists a valuation $w \colon \Z H/\IP \to \R_{\infty}$ such that 
    $w|_{\Z} = v$  and $w|_H = \chi $. 
Note that $w^{-1}(\infty)$ is a prime ideal of $\Z H/\IP$. Let $\K$ be the fractional field of the quotient domain  $(\Z H/\IP)/w^{-1}(\infty)$. 
Then $w$ can be viewed as a valuation on $\K$ such that $w|_H=\chi$.

    Consider the following composition of maps  
    $$G\twoheadrightarrow H \hookrightarrow \Z H \twoheadrightarrow \Z H/\IP \to \K,$$ 
    which gives a multiplicative character $ \rho \colon G \to \K^* $. Let $L_\rho$ be the corresponding rank one local system on $X$ with coefficient $\K$. We claim that \begin{center}
          $H_i(X; L_\rho)\neq 0$ for some $0\leqslant i \leqslant k$.
     \end{center}
   In fact, $H_*(X; L_\rho)$  can be computed by the chain complex $C_*(X^H; \Z)\otimes_{\Z H} \K,$ where $\K$ is viewed as a $\Z H$-module associated to the representation $\rho$. 
Note that  $$\bigcap\limits_{0\leqslant j\leqslant k}\mathcal{J}^j(X; \Z)=\cJ^{\leqslant k}(X; \Z) \subseteq \sqrt{ \cJ^{\leqslant k}(X; \Z)} \subseteq \IP.  $$ Then there exists some $0\leqslant i \leqslant k$ such that $ \mathcal{J}^i(X; \Z) \subseteq \IP$.    By the construction of the field $\K$, we have that $\cJ^{i}(X; \Z)\otimes_{\Z H}\K=0,$ hence $H_i(X; L_\rho)\neq 0$.

Since $\chi=w\circ \rho$ is non-zero,  \cref{thm PS} shows that $\chi \notin \Sigma^k(X; \Z)$, hence $$\mathrm{S}( \Delta^v_\IP(H)) \subseteq \Sigma^k(X; \Z)^{c}$$ for any non-negative valuation $v$ on $\Z$ and any minimal prime ideal $\IP $ containing $\sqrt{ \cJ^{\leqslant k}(X; \Z)}$. Therefore, the first inclusion in \cref{main inclusion} follows from \cref{thm:BG_main_theorem,prop trop equals Sigma complement}.
The remaining properties are a direct consequence of \cref{thm BG dense}.
\qed



 \begin{remark}\label{rem compare Suciu}
    Let us explain Suciu's work \cite[Theorem 1.1]{Suciu21} in our notations. 
    Assume that $\bbmk$ is algebraically closed. 
Since $\Trop_{\bbmk}(\mathcal{J}^{\leqslant k}(X; \bbmk))$ depends only on $ \sqrt{\mathcal{J}^{\leqslant k}(X; \bbmk)}$
    and $\sV^{\leqslant k}(X;\bbmk)$ is the variety of the ideal $\mathcal{J}^{\leqslant k}(X; \bbmk) $, one can also denote  $\Trop_{\bbmk}(\mathcal{J}^{\leqslant k}(X; \bbmk))$ by $\Trop_{\bbmk}(\sV^{\leqslant k}(X; \bbmk))$. 
    When $\bbmk=\C$, due to \cref{rem Suciu} and \cite[Section 4]{Suciu21}, $\Trop_{\C}(\sV^{\leqslant k}(X; \C))$ is exactly the tropical variety considered in \cite[Theorem 1.1]{Suciu21}. On the other hand, if $\bbmk$ has positive characteristic, using the valued field $ \bbmk((t^\Q))$  one can use Suciu's proof to show that  $\Sigma^k(X;\Z)\subseteq  \mathrm{S}\big(\Trop_{\bbmk}(\sV^{\leqslant k}(X; \bbmk))\big)^c .$
\end{remark}

\subsection{Comparing with Bieri, Groves and Strebel's results}\label{subsection comparing BGS}
When $k=1$, \cref{thm:compare_two_bnsr} is essentially due to Bieri, Groves, and Strebel's work \cite{BS80,BS81,BG84}. We explain it in detail in this subsection.

Let $G$ be a finitely presented group. Set $G'=[G,G]$, and $G''=[G',G']$. 
Then $G/G''$ is the maximal metabelian quotient group of $G$, and $\mathrm{S}(G)=\mathrm{S}(G/G'')$. 
By \cref{prop naturality}, we have $$\Sigma^1(G; \Z)\subseteq \Sigma^1(G/G''; \Z).$$
Consider the short exact sequence
$$1\to G'\slash G''\to G\slash G''\to H\to 1,$$
with $H$ acting on $G'\slash G''$ by conjugation. Since $G$ is finitely presented, $G'/G''$ is finitely generated as a $\Z H$-module. 
In fact, if $X$ is a connected CW complex such that $\pi_1(X)=G$, then $G'/G''$ is the first Alexander invariant $H_1(X^H; \Z)$.  For the metabelian group $G/G'' $, 
  Bieri and Strebel showed in \cite{BS80} (see also \cite{BNS}) that $$ \Sigma^1(G/G''; \Z)=  \Sigma(H_1(X^H; \Z)),$$ where $\Sigma(H_1(X^H; \Z))$ is the Sigma invariant defined in  \cref{sec:background}.
By  \cref{prop trop equals Sigma complement}, we have
$$   \Sigma(H_1(X^H; \Z))= \mathrm{S}\big(\Trop_\Z (\Ann(H_1(X^H; \Z)))\big)^c.$$

Let $I$ be the augmentation ideal of the group ring $\Z H$, i.e. $I=\langle h-1\mid h\in H\rangle$. We have $\mathrm{S}(\Trop_\Z(I))=\emptyset$ and $I=\sqrt{\Ann(H_0 (X^H; \Z))}$. It follows from \cref{thm PS two types of ideals} that
$$\mathrm{S}\big(\Trop_\Z (\Ann (H_1(X^H; \Z)))\big)=\mathrm{S}\big(\Trop_\Z (\Ann (H_{\leqslant 1}(X^H; \Z)))\big)=\mathrm{S}\big(\Trop_{\Z}(\mathcal{J}^{\leqslant 1}(X; \Z))\big).$$
Putting all together, we get that $$\Sigma^1(G; \Z)\subseteq \Sigma^1(G/G''; \Z)=  \mathrm{S}\big(\Trop_{\Z}(\mathcal{J}^{\leqslant 1}(X; \Z))\big)^c ,$$
which is exactly the first inclusion in \cref{main inclusion} for $k=1$. In particular, if $G$  itself is metabelian, then the inclusion becomes an equality (for $k=1$).

\section{Examples and application to Dwyer-Fried sets} \label{section examples}
\subsection{Examples}
In this subsection, we compute various examples.
In the first three examples, we compare \cref{thm:compare_two_bnsr} with \cite[Theorem 1.1]{Suciu21}. These three examples are all one-relator groups with two generators, whose BNS-invariants can be calculated according to Brown's algorithm \cite[Section 4]{Brown}. Notice here we use the invariant $\Sigma^1(G;\Z)$, which is $-\Sigma^1(G)$. 


\begin{example}\label{example 1}
 Let   $G=\langle a,b\mid aba^{-1}b^{-2}\rangle$ be the Baumslag–Solitar group $\mathrm{BS}(1,2)$. This is the example considered in \cite[Example 8.2]{Suciu21}.
Let $\chi\colon G\to \R$ be the non-zero homomorphism such that $\chi(a)=1$ and $\chi(b)=0$. Then $\mathrm{S}(G)=\{\pm \chi\}$.
The abelianization of the Fox derivative gives $x-2$. Then one gets  
$\sqrt{\cJ^{\leqslant 1}(G; \Z)}=(x-1)\cap (x-2)$ and 
$\sV^{\leqslant 1}(G; \C)=\{1,2\}$. 
Hence \begin{center}
    $\mathrm{S}(\Trop_\Z( \sqrt{\cJ^{\leqslant 1}(G; \Z)}))=\{ \chi\}$ and  $\mathrm{S}(\Trop_\C( \sV^{\leqslant 1}(G; \C))=\emptyset$.
\end{center}
Calculation directly gives that $\Sigma^1(G; \Z)=\{-\chi\}.$
\end{example}

\begin{example}\label{example 2}
Let   $G=\langle a,b\mid a^{-1}b^{-1}ab^2a^{-1}b^{-1}a^2 b^{-1}a^{-1}ba^{-1}bab^{-1}\rangle$. This is the example in \cite[Page 492]{Brown} and in \cite[Example 8.5]{Suciu21}.
The abelianization of the Fox derivative gives $$\begin{pmatrix}
  (x^{-1}_2-1)(x^{-1}_1-1)  & -x_1^{-1}x_2^{-1}(x_1-1)^2
\end{pmatrix}.$$
Then one gets  
$\sqrt{\cJ^{\leqslant 1}(G; \Z)}=(x_1-1)$ and 
$\sV^{\leqslant 1}(G; \C)=\{x_1=1\}$. 
Hence \begin{center}
    $\mathrm{S}(\Trop_\Z( \sqrt{\cJ^{\leqslant 1}(G; \Z)}))=\{ (0,1),(0,-1)\}=\mathrm{S}(\Trop_\C( \sV^{\leqslant 1}(G; \C))$.
\end{center}
 Calculation via Brown's algorithm gives that $\Sigma^1(G; \Z)$ consists of two open arcs on the unit circle, joining the points $(1,0)$ to $(0,1)$ and $(0,1)$ to $(-\frac{\sqrt{2}}{2},-\frac{\sqrt{2}}{2})$. 
So  the first inclusion in \cref{main inclusion}  is strict in this case, as shown below in \cref{fig:example_2}.
\end{example}

\begin{figure}[ht!]
    \centering
    \begin{subfigure}[b]{0.45\textwidth}
        \centering
        \begin{tikzpicture}[scale=0.4]
            \draw[-,dashed] (-3,0) -- (3,0);
            \draw[-,dashed] (0,-3) -- (0,3);
            \draw[blue,thick] (0,2) circle (4pt);
            \draw[blue,thick] (2,0) circle (4pt);
            \draw[blue,thick] (-1.414213562, -1.414213562) circle (4pt);
            \draw (0,0) circle (2);
            \fill[blue] (-2,0) circle (4pt);
            \draw[blue, thick] 
                (2,0) arc[start angle=0, end angle=225, radius=2];
        \end{tikzpicture}
        \caption{$\Sigma^1(G;\Z)$.}
        \label{fig:example2_BNS}
    \end{subfigure}
    \hfill
    \begin{subfigure}[b]{0.45\textwidth}
        \centering
        \begin{tikzpicture}[scale=0.4]
            \draw[-,dashed] (-3,0) -- (3,0);
            \draw[-,dashed] (0,-3) -- (0,3);
            \draw[blue,thick] (0,2) circle (4pt);
            \draw[blue,thick] (0,-2) circle (4pt);
            \draw[blue, thick] 
                (0,-2) arc[start angle=-90, end angle=90, radius=2];
            \draw[blue, thick] 
                (0,2) arc[start angle=90, end angle=270, radius=2];
        \end{tikzpicture}
        \caption{$\Z(\text{or\ }\C)$-Tropical upper bound.}
        \label{fig:example2_bound}
    \end{subfigure}
    \caption{\cref{example 2}.}
    \label{fig:example_2}
\end{figure}

\begin{example}\label{example 3}
As in \cite[Example 3.5]{PapaSuciu10} and \cite[Lemma 10.3]{SuciuYZ}, given a Laurent polynomial $f(x_1,x_2)\in \Z[x_1^{\pm 1},x_2^{\pm 1}]$, there exists  a group $G$ with two generators and one relation such that its abelianization is $\Z^2$ and the abelianization of the Fox derivative is 
 $$\begin{pmatrix}
 f\cdot (x_2-1)  & -f\cdot (x_1-1)
\end{pmatrix} .$$
It implies that 
$\sqrt{\cJ^{\leqslant 1}(G,\Z)} =(f)\cap (x_1-1,x_2-1) $.
Set $f(x_1,x_2)=x_1+x_2-2$. Then \cref{thm:compare_two_bnsr} gives 
$$\Sigma^1(G; \Z) \subseteq  \mathrm{S}(\Trop_{\Z} (x_1+x_2-2))^c.$$ 
Suciu showed that \cite[Theorem 1.1]{Suciu21}
$$\Sigma^1(G; \Z) \subseteq  \mathrm{S}(\Trop_{\mathbb{C},v_0} (x_1+x_2-2))^c.$$ 
The same proof is indeed valid for any field coefficients, hence one also gets $$\Sigma^1(G; \Z) \subseteq \mathrm{S}(\Trop_{\mathbb{F}_2,\hat{v}_2} (x_1+x_2-2))^c.$$ 
By \cref{example trop comparision} and \cref{fig:trivial_and_mod p,fig:all_three}, we have
$$\mathrm{S}(\Trop_{\mathbb{F}_2,\hat{v}_2} (x_1+x_2-2))\bigcup\mathrm{S}(\Trop_{\mathbb{C},v_0} (x_1+x_2-2)) \subsetneq \mathrm{S}(\Trop_{\Z} (x_1+x_2-2)).$$
\end{example}

We end this subsection with the computations for compact Riemann orbifold groups. 

\begin{definition}  Let $C_{g}$ be a compact Riemann surface of genus $g\geqslant 1$ and let $s\geqslant 0$ be an integer. If $s>0$, fix points $\{q_1,\ldots,q_s\}$ in $C_g$ and  assign to these points an integer weight vector ${\bm \mu}=(\mu_1,\cdots,\mu_s)$ with $\mu_i\geqslant 2$. The orbifold Euler characteristic of the
surface with marked points is defined as $$\chi^\orb(C_g,{\bm \mu})=2-2g- \sum_{j=1}^s (1-\dfrac{1}{\mu_j}).$$
     The orbifold group $\pi_1^{\orb}(C_{g},{\bm \mu})$ associated to these data has presentation as follows: 
\[ \langle x_{1},\ldots,x_{g},y_{1},\ldots,y_{g},z_{1},\ldots,z_{s}|\prod_{i=1}^{g}[x_{i},y_{i}]\cdot\prod_{j=1}^{s} z_{j}=1,z_{j}^{\mu_{j}}=1 \text{ for all } 1\leqslant j\leqslant s\rangle. \]
\end{definition}

If $\chi^\orb(C_g,{\bm \mu})=0$, i.e. $ g=1$ and $s=0$, then $\pi_1(C_g)\cong \Z^2$. Hence  $\sV^1(\Z^2;\bbmk)=\{1\}$ for any algebraically closed field $\bbmk$ and $\Sigma^1(\pi_1(C_g);\Z)=S^1$.
So we focus on the case $\chi^\orb(C_g,{\bm \mu})<0$, i.e. either $g>1$ or $g=1$ and $s>0$. 
 Set $\theta({\bm \mu})=\frac{\prod_{j=1}^s \mu_j}{\lcm(\mu_1,\cdots,\mu_s)}$. Then $H=H_1(\pi_{1}^{\orb}(C_{g},{\bm \mu}); \Z)$ has free abelian part $\Z^{2g}$ and its torsion part has order $\theta({\bm \mu})$. Next we compute the BNS invariant and homology jump loci of $\pi_{1}^{\orb}(C_{g},{\bm \mu}) $. 
 
\begin{proposition}\label{prop compact orbifold group}
Let $\bbmk$ be an algebraically closed field with $\mathrm{char}(\bbmk)=p\geqslant 0$.
With the above assumptions and notations, if $\chi^\orb(C_g,{\bm \mu})<0$, we have 
\[
\sV^1(\pi_{1}^{\orb}(C_{g},{\bm \mu}); \bbmk)  = 
\begin{cases}
\Hom(H;\bbmk^*),   &  \text{if } g > 1, \text{ or } g = 1 \text{ and } p \\
                   & \text{divides some } \mu_j, \\[5pt]
\Hom(H;\bbmk^*)' \cup \{1\},   &  \text{if } g = 1, \theta({\bm \mu}) > 1, \\
                   & \text{and } p \text{ does not divide any } \mu_j, \\[5pt]
\{1\},             &  \text{if } g = 1, \theta({\bm \mu}) = 1, \\
                   & \text{and } p \text{ does not divide any } \mu_j, 
\end{cases}
\]
 where $\Hom(H;\bbmk^*)'$ is  $\Hom(H;\bbmk^*)$ taking out the connected component containing $1$, and we use the convention that $0$ does not divide any nonzero integer.  Then 
   $$\Trop_\bbmk(\sV^1(\pi_{1}^{\orb}(C_{g},{\bm \mu});\bbmk))=\begin{cases}
\{0\},   &   \text{ if } g=1, \theta({\bm \mu})=1, \text{ and } p \text{ does not divide any } \mu_j, \\
 \R^{2g}, & \text{ otherwise.} 
 \end{cases}
 $$
Hence   $\Sigma^1(\pi_{1}^{\orb}(C_{g},{\bm \mu});\Z)=\emptyset .$
\end{proposition}
\begin{proof}
    When $\mathrm{char}(\bbmk)=0$, the computations for $\sV^1(\pi_{1}^{\orb}(C_{g},{\bm \mu});\bbmk)$ can be found  in \cite[Section 2]{ACM} or \cite[Section 10]{Suciu21}.
    When $\mathrm{char}(\bbmk)=p>0$, by the Fox calculus presented in \cite[Proof of Proposition 2.11]{ACM}, one can compute $\sV^1(\pi_{1}^{\orb}(C_{g},{\bm \mu});\bbmk)$ as in \cite[Section 3.2]{LiLiu}. 
    Then the computations for tropical variety follow easily. 
Hence $\Sigma^1(\pi_{1}^{\orb}(C_{g},{\bm \mu});\Z)=\emptyset $ by \cref{thm:compare_two_bnsr}.
\end{proof}

\begin{remark}\label{remark orbifold group}
In the above proof, if  $\chi^\orb(C_g,{\bm \mu})<0$, we have
$$
\Sigma^1(\pi_{1}^{\orb}(C_{g},{\bm \mu});\Z) =\begin{cases}
\mathrm{S}( \Trop_\bbmk(\sV^1(\pi_{1}^{\orb}(C_{g},{\bm \mu});\bbmk)))^c,   &  \text{if } g = 1, \theta({\bm \mu}) = 1, \\
                   & \text{and } p \text{ divides some } \mu_j \\[5pt]
   \mathrm{S}(\Trop_\C (\sV^1(\pi_{1}^{\orb}(C_{g},{\bm \mu});\C)))^c, &  \text{ otherwise.}
\end{cases} 
$$
This shows that the $p$-adic tropicalization considered in \cref{prop three trop union Z} is not needed for $\Trop_\Z(\cJ^1( \pi_{1}^{\orb}(C_{g},{\bm \mu});\Z))$.
\end{remark}
 

\subsection{Dwyer-Fried sets}
\label{subsec:DF_set}

Consider an epimorphism $\nu: \pi_1(X)=G\twoheadrightarrow H'$ with $H'$ abelian. Let $X^{\nu}\to X$ be the corresponding abelian covering. Consider $\bbmk$ an algebraically closed field and 
denote by $\nu^\#\colon \mathrm{Hom}(H'; \mathbbm{k}^*)\to \mathrm{Hom}(G; \bbmk^*)$ the induced morphism.
Suciu, Yang, and Zhao generalized Dwyer and Fried's results (from a free abelian quotient to an abelian quotient) in \cite[Theorem B]{SuciuYZ} as follows:
$$  \mathrm{dim}_{\mathbbm{k}}H_{\leqslant k}(X^{\nu};\mathbbm{k})< \infty \Longleftrightarrow \mathrm{im}(\nu^{\#})\cap \sV^{\leqslant k}(X; \bbmk)\text{\ is finite},$$
where $H_{\leqslant k}(X^{\nu};\mathbbm{k})=\bigoplus_{i=0}^{k}H_i(X^{\nu};\mathbbm{k})$. 
With our notations,  this conclusion can be partially 
generalized  to the integer coefficients as follows. 

\begin{proposition} \label{prop DF}
Let $\nu: G\twoheadrightarrow H'$ be the quotient from $G$ to an abelian group $H'$ with $\rank_\Z H'\geqslant 1$. 
Consider $H_i(X^{\nu}; \Z)$ as a finitely generated $\Z H'$-module induced by the deck transformation.
 Then $H_i(X^{\nu}; \Z)$  is finitely generated over $\Z$ if and only if $\Trop_\Z (\Ann (H_i(X^\nu;\Z)))\subseteq \{0\}$.
Moreover, if $\Trop_\Z (\cJ^{\leqslant k}(X; \Z)) \subseteq \{0\}$, then $H_i(X^{\nu}; \Z)$  is finitely generated over $\Z$ for any such abelian cover $\nu$ and all $0\leqslant i \leqslant k$. 
\end{proposition}
\begin{proof}
The first claim follows directly from \cref{thm:BS_finitely_generated_module}  and \cref{prop trop equals Sigma complement}.   
For the second claim,  since $H'$ is abelian, the map $\nu$ factors through the abelianization map $G\twoheadrightarrow H$ with an epimorphism $\psi\colon H\twoheadrightarrow H'$. Then $\psi$ induces   an embedding
 $$\psi^*\colon \mathrm{Hom}(H'; \R) \hookrightarrow \mathrm{Hom}(H; \R)$$
and a ring epimorphism $\Z H \twoheadrightarrow \Z H'$. Similar to the discussion at the end of \cref{sec:tropical}, there exists an ideal $K\subseteq \Z H$ such that $\Z H/K\cong \Z H'.$ Set $I=\sqrt{\cJ^{\leqslant k}(X; \Z)} \subseteq \Z H$ and $I'=\sqrt{\Ann (H_{\leqslant k}(X^{\nu}; \Z))}\subseteq \Z H'.$
Then we claim that $I'=I\otimes_{\Z H} \Z H'$, which is the image of $I$ in the quotient ring  $\Z H'$. In fact, one can also define the jump ideal
 $\cJ^{\leqslant k}(X,\nu; \Z) $ for the cellular chain complex $C_*(X^\nu;\Z)$ as a complex of $\Z H'$-modules as in \cref{defn:jump ideal}. Then by \cref{thm PS two types of ideals} we have $\sqrt{\cJ^{\leqslant k}(X,\nu; \Z)}=I' $. On the other hand, by definition we have $\cJ^{\leqslant k}(X,\nu; \Z) = \cJ^{\leqslant k}(X; \Z) \otimes_{\Z H} \Z H' $, hence  $I' = I\otimes_{\Z H} \Z H'$.
Note that $I+K$ is the unique ideal in $\Z H$ such that $\Z H/(I+K) \cong \Z H'/I'$. Then we have  $$\psi^*(\Trop_\Z(I'))=\Trop_\Z(I+K)\subseteq \Trop_\Z(I),$$
where the first equality follows from \cref{prop functor}.
Hence $ \Trop_\Z(I)\subseteq \{0\}$ implies  $\Trop_\Z(I') \subseteq\{0\}$. Then the second claim follows from \cref{thm:BS_finitely_generated_module}  and \cref{prop trop equals Sigma complement}.   
\end{proof}
\begin{remark}  
Since $\Trop_\Z(K)=\psi^*(\mathrm{Hom}(H',\R))$, we have 
the inclusion $$\psi^*(\Trop_\Z(I'))=\Trop_\Z(I+K) \subseteq \Trop_\Z(I)\cap \Trop_\Z(K) ,$$ which could be strict. For example, if $I$ is a proper ideal in $\Z H$ and $I'=\Z H'$, then $0\in \Trop_\Z(I)$ and $\Trop_\Z(I')=\emptyset $.
This happens since tropicalization does not always respect intersections even over field coefficients: 
if $V$ and $W$ are subvarieties of $(\bbmk^*)^n$, then $\Trop_\bbmk(V \cap W)\subseteq \Trop_\bbmk(V)\cap \Trop_\bbmk(W)$, but the inclusion may be strict. 
\end{remark}

\begin{example} Consider the group $G$ as in \cref{example 1}. 
    By the Fox calculus there, we have the following 
    isomorphisms
    $$G'/G''\cong  \Z[x^{\pm 1}]/(x-2) \cong \Z[\frac{1}{2}].$$
    In particular, $\Z[\frac{1}{2}] $ is not finitely generated as $\Z$-module, which correspondences to $\Trop_\Z( \Ann(G'/G''))=\{\lambda \cdot \chi \mid \lambda \in \R_{\geqslant 0}\}$. 
    On the other hand, for any field $\bbmk$,  $G'/G''\otimes_\Z \bbmk$ is finite dimensional, which correspondences to $\Trop_\bbmk( \Ann(G'/G''\otimes_\Z \bbmk))=\{0\}$. 
\end{example}

\section{K\"ahler groups}\label{sec:applications}

\subsection{General results on K\"ahler groups}
\label{subsec_K\"ahler}
Recall that a group $G$ is called a K\"ahler group if it can be realized as the fundamental group of a compact K\"ahler manifold. Using Simpson's Lefschetz theorem \cite[Theorem 1]{Simpson93}, Delzant gave a complete description of $\Sigma^1(G)$ for $G$ a K\"ahler group in \cite[Theorem 1.1]{Delzant10}. 
To explain Delzant's results, we first recall the definition of orbifold fibrations.
\begin{definition}
Let $X$ be a compact K\"ahler manifold with $\pi_1(X)=G$. A holomorphic map $f\colon X \to C_{g}$ is called an orbifold fibration if $f$ is surjective with connected fibers onto a Riemann surface $C_g$ with genus $g\geqslant 1$. Assume that $f$ has multiple fibers over the points  $\{q_1,\ldots,q_s\}$ in $C_g$ and let $\mu_j$ denote the multiplicity of  the multiple fiber $f^{*}(q_j)$  (the $\gcd$ of the coefficients of the divisor $f^* q_j$). 
Such an orbifold fibration is denoted  as $f\colon X \to (C_{g},{\bm \mu})$ and is called hyperbolic if  $\chi^\orb(C_g,{\bm \mu})<0$.
\end{definition}
Two orbifold fibrations $f\colon X \to (C_{g},{\bm \mu})$ and $f'\colon X \to (C_{g'},{\bm \mu}')$  are equivalent
if there is a biholomorphic map $h\colon C_g \to C_{g'}$ which sends marked points to marked
points while preserving multiplicities.  A compact Kähler
manifold $X$ admits only finitely many equivalence classes of hyperbolic orbifold fibrations, see e.g. \cite[Theorem 2]{Delzant08},.
\begin{theorem}  [Theorem 1.1, \cite{Delzant10}] \label{thm Delzant} Let $X$ be a compact K\"ahler manifold with $\pi_1(X)=G$. Then we have

$$\Sigma^1(G;\Z) = \mathrm{S}\big( \bigcup_f \mathrm{im}(f^*\colon H^1(C_g;\R)\to H^1(X;\R) ) \big)^c,$$
where the union runs over all hyperbolic orbifold fibrations of $X$. In particular, it is a finite union. 
Moreover, $\Sigma^1(G;\Z)=\emptyset$ if and only if there exists a hyperbolic orbifold fibration $f\colon X \to C_g$ such that $f^*\colon H^1(C_g;\R)\to H^1(X;\R)  $ is an isomorphism.
\end{theorem}

A hyperbolic orbifold fibration $f\colon X \to (C_{g},{\bm \mu})$   induces an epimorphism from $G$ to the orbifold group
$$ f_*\colon G \twoheadrightarrow \pi_1^{\orb}(C_{g},{\bm \mu}).$$
Hence it induces an embedding  (see e.g. \cite[Lemma 2.13]{Suciu14abelian_cover})
$$\sV^1(\pi_{1}^{\orb}(C_{g},{\bm \mu});\bbmk) \hookrightarrow \sV^1(G;\bbmk) .$$
for any algebraically closed field coefficient $\bbmk$. Then we have $$\Trop_\bbmk (\sV^1(\pi_{1}^{\orb}(C_{g},{\bm \mu});\bbmk)) \subseteq \Trop_\bbmk(\sV^1(G;\bbmk)). $$
Hence   \cref{thm:compare_two_bnsr}, \cref{thm Delzant} and \cref{remark orbifold group} give the following observation. 
\begin{proposition}\label{prop:abelian_bnsr_K\"ahler}
Let $G$ be a K\"ahler group. Then we have $$\Sigma^1(G;\Z)=\mathrm{S}\big(\Trop_\Z( \cJ^1(G;\Z))\big)^c=\mathrm{S}\big( \bigcup_{\mathrm{char}(\bbmk)=p\geqslant 0}\Trop_{\bbmk}(\sV^{1}(G; \bbmk)) \big)^c,$$
where the last union runs over algebraically closed coefficient field $\bbmk $ with $\mathrm{char}(\bbmk)=p\geqslant 0$.
\end{proposition}
Related results of this proposition are first observed by
Papadima and Suciu in \cite[Theorem 16.4]{PapaSuciu10}, later improved by Suciu in \cite[Theorem 12.2]{Suciu21}.



Denote the derived series of $G$ as $G'= \mathcal{D}^1 G =[G,G]$, and $\mathcal{D}^{m+1}G=[\mathcal{D}^m G, \mathcal{D}^m G]$ for $m\geqslant 1$, and the solvable quotient as $\mathcal{R}^{m} G=G\slash \mathcal{D}^{m}G$. 
By \cref{prop:abelian_bnsr_K\"ahler}, we have the following observation, which may give new restrictions for the K\"ahler group.
\begin{corollary}\label{cor:sigma_metabelian_quotient_K\"ahler}
Let $G$ be a K\"ahler group and $G/G''$ be its maximal metabelianization. Let $N$ be a normal subgroup of $G$ with $N\leqslant G''=\mathcal{D}^2 G$. Then we have 
$$\Sigma^1(G; \Z)=\Sigma^1(G/N; \Z)=\Sigma^1(G/G''; \Z).$$
In particular,  $\Sigma^1(G;\Z)=\Sigma^1(\mathcal{R}^m G;\Z)$ for any $m\geqslant 2$.
\end{corollary}
\begin{proof}
Since $\mathrm{S}(G\slash G'')=\mathrm{S}(G)$, the naruality property in \cref{prop naturality} gives us $$\Sigma^1(G;\Z)\subseteq 
\Sigma^1(G/G'';\Z).$$

Since $\Trop_\Z (\cJ^1(G; \Z))$ depends only on
$\sqrt{\cJ^1(G; \Z)}$  and   $\sqrt{\cJ^1(G; \Z)}$ depends only on $G/G''$ (see e.g. \cite[Section 2]{Suciu14abelian_cover}), we have 
$$ \Trop_\Z (\cJ^1(G; \Z))=\Trop_\Z (\cJ^1(G/G''; \Z)).$$
On the other hand, for metabelian group $G/G''$,  by \cref{thm:compare_two_bnsr} we have 
$$\Sigma^1(G/G''; \Z)\subseteq \mathrm{S}\big(\Trop_\Z (\cJ^1(G/G'';\Z))\big)^c.$$
Putting all together, we get that 
$$\Sigma^1(G; \Z) \subseteq \Sigma^1(G/G'';\Z)\subseteq \mathrm{S}\big(\Trop_\Z (\cJ^1(G/G'';\Z))\big)^c= \mathrm{S}\big(\Trop_\Z (\cJ^1(G;\Z))\big)^c.$$
Then the claim follows from  \cref{prop:abelian_bnsr_K\"ahler} and \cref{prop naturality}. 
\end{proof}


\begin{proof}[Proof of \cref{prop K\"ahler}]
To simplify the notations, set $Q=G/G''$ and $Q'=G'/G''$.
We prove the proposition by four steps.

\noindent Step 1: We first prove the equivalence of $(i)-(iv)$.  

\noindent $(i)\iff (ii)$ and $(iii)\iff (iv) $: Both follow from \cref{thm BNS}.

\noindent $(i)\iff (iii)$: It follows from \cref{cor:sigma_metabelian_quotient_K\"ahler} and $\mathrm{S}(G)=\mathrm{S}(Q)$.

\noindent Step 2: We prove  $(iv)\Rightarrow (v)\Rightarrow (vi) \Rightarrow (iii)$.

\noindent $(iv)\Rightarrow (v)$: $Q$ is an extension of $Q'$  by $G/G'$. By the assumption, both $Q'$ and $G/G'$ are finitely generated abelian groups, particularly polycyclic. Then by \cite[Proposition 13.73(5)]{DK}, $Q$ is polycyclic.

\noindent $(v)\Rightarrow (vi)$: By \cite[Proposition 13.84]{DK}, every finitely generated polycyclic group is  finitely presented.

\noindent $(vi)\Rightarrow (iii)$:
Note that $Q$ is a finitely generated metabelian group. By a nice theorem due to Birei and Strebel \cite[Theorem A]{BS80}, we have that 
    \begin{center}
        $Q$ is finitely presented if and only if $\Sigma^1(Q;\Z) \cup -\Sigma^1(Q;\Z)=S(Q).$
    \end{center} 
Note that by \cref{cor:sigma_metabelian_quotient_K\"ahler} and \cref{thm Delzant}, we have $\Sigma^1(Q;\Z)=-\Sigma^1(Q;\Z)$. Then the claim follows. 

\noindent Step 3: We prove  $(iv)\Rightarrow  (vii)\Rightarrow (vi)$. 


\noindent $(iv)\Rightarrow (vii)$: If $Q'$ is finitely generated,
it follows from \cite[Lemme 3.1]{Delzant10} or \cite[Corollary 3.6]{Burger} that  $Q$ is virtually nilpotent. 

\noindent $(vii)\Rightarrow (vi)$: By assumption, there exists a normal subgroup $N$ of $Q$ with a short exact sequence 
$$1\to N \overset{q}{\to} Q \to Q/N\to 1,$$ where $N$ is nilpotent and $Q/N$ is finite. 
Since $Q$ is finitely generated, so is $N$ by \cite[Lemma 7.85]{DK}. Note that finitely generated nilpotent groups and finite groups are both finitely presented, hence $Q$ is also finitely presented by \cite[Proposition 7.30]{DK}.
 
\noindent Step 4: Finally we show that  $(iv)\Rightarrow  (viii)\Rightarrow (ix)\Rightarrow (iii)$. 

\noindent $(iv)\Rightarrow  (viii)$: It is obvious.

\noindent $(viii)\Rightarrow (ix)$:  Let $\bbmk$ an algebraically closed field.  By a fact in commutative algebra (see e.g. \cite[Proposition 9.3]{SuciuYZ}), $Q'\otimes_\Z\bbmk$ being finite-dimensional implies that the variety of the ideal $\Ann (Q'\otimes_\Z \bbmk)$ consists of at most finitely many points. Note that $Q'\otimes_\Z\bbmk$ is the first Alexander invariant of $G$ with $\bbmk$-coefficient. Then by \cref{thm PS two types of ideals}, $\sV^1(G;\bbmk)$ consists of  finitely many points.

\noindent $(ix)\Rightarrow (iii)$: 
Note that \cref{thm structure trop} shows that taking tropicalization preserves dimension over field coefficients and $ \Trop_\K(\sV^1(G;\bbmk))$ is homogeneous with respect to scalar multiplication by a positive real number. 
If $\sV^1(G;\bbmk)$ consists of finitely many points, we have $\Trop_\bbmk(\sV^1(G;\bbmk))\subseteq \{0\}$, 
hence  $\Trop_\Z(\cJ^1(Q,\Z))\subseteq \{0\}$ due to \cref{prop:abelian_bnsr_K\"ahler}. Therefore,  by the discussion in \cref{subsection comparing BGS}, we have $\mathrm{S}(Q)=\Sigma^{1}(Q;\Z)$. 
\end{proof}







\subsection{Weighted right-angled Artin group} \label{subsec WRAAG}
In this subsection, we classify the K\"ahler weighted right-angled Artin groups. 

\begin{example} \label{example two graphs}
    Let $\Gamma_\ell= (V, E, \ell)$ and $\Gamma'_{\ell'}= (V', E', \ell')$ be two labeled graphs. Denote by $\Gamma_\ell \sqcup \Gamma'_{\ell'} $ their disjoint union. Denote by $\Gamma_\ell * \Gamma'_{\ell'}$ their join, with vertex set $V\sqcup V'$, edge set $E\sqcup E' \sqcup \{\{a,a'\}|a\in V, a'\in V'\}$ and  weight 1 on all the joining edges $\{a,a'\}$.
 Then we have 
 \begin{center}
      $G_{\Gamma_\ell \sqcup \Gamma'_{\ell'}}=G_{\Gamma_\ell} * G_{\Gamma'_{\ell'}} $ and  $G_{\Gamma_\ell * \Gamma'_{\ell'}}=G_{\Gamma_\ell} \times G_{\Gamma'_{\ell'}} .$
 \end{center}
 \end{example}

For the edge weighted graph $\Gamma_\ell$, if we forget the weight on the edges, we get the finite simple graph $\Gamma$.  Let $G_\Gamma$ denote the corresponding right-angled Artin group. 
The jump loci of $G_\Gamma$ are completely characterized by Dimca, Papadima, and Suciu as follows (see also \cite[Theorem 5.5, Corollary 5.6]{PS06} for the corresponding results about resonance varieties). 

\begin{theorem}[\cite{DPS}, Proposition 11.5] Let $\Gamma$ be a finite simple graph and let $G_\Gamma$ denote the corresponding right-angled Artin group. Then we have 
$$\sV^1(G_\Gamma; \mathbb{C}) = \bigcup_{\substack{W \subseteq V \\ \Gamma_W \text{ is maximally disconnected}}} \mathbb{T}_W,
$$
where the union is taken over all subsets $W \subseteq V$ such that the subgraph $\Gamma_W$ is maximally disconnected, i.e. $\Gamma_{W}$ is disconnected and there is no disconnected subgraph of $\Gamma$ strictly containing $\Gamma_{W}$. Here for $W\subseteq V$, $\mathbb{T}_W \subseteq \mathbb{T}_V=(\C^*)^{|V|}$ is given by
$$ \mathbb{T}_W =\{ (x_a)_{a\in V}\in (\C^*)^{|V|} \mid x_a=1 \text{ for } a \notin W\}.$$
    In particular, $\sV^1( G_\Gamma;\C)=(\C^*)^{|V|} $ if and only if $\Gamma$ is disconnected.
\end{theorem}

By an observation, we show that $\sV^1(G_{\Gamma_\ell};\C)$ is same as $\sV^1( G_\Gamma;\C)$.
\begin{corollary} \label{cor disconnected}
Let $G_{\Gamma_\ell}$ denote the weighted right-angled Artin group associated to an edge weighted graph $\Gamma_\ell$. 
Let $G_{\Gamma}$ denote the right-angled Artin group associated to the corresponding simple graph $\Gamma$ by forgetting the weights on edges. 
There is a natural isomorphism $H_1(G_{\Gamma_\ell};\Z) \cong H_1(G_{\Gamma};\Z) $, which induces an isomorphism $\mathrm{Hom}(G_{\Gamma_\ell},\C^*) \cong \mathrm{Hom}(G_{\Gamma},\C^*) $. 
Under this isomorphism, we have 
    $$ \sV^1(G_{\Gamma_\ell};\C)\cong\sV^1( G_\Gamma;\C). $$
     In particular, $\sV^1(G_{\Gamma_\ell};\C)=(\C^*)^{|V|} $ if and only if $\Gamma_\ell$ is disconnected.
\end{corollary} 
\begin{proof}
    The degree one jump loci can be computed by the Alexander matrix given by the Fox calculus. For an edge  with weight $m$ in $G_{\Gamma_\ell}$ connecting vertices  $a$ and $b$, we have the relation $[a,b]^m=1$ in $G_{\Gamma_\ell}$. 
    The abelianization of the Fox derivative of this relation gives $$\begin{pmatrix}
  m(1-b)  & m(a-1) & 0 & \cdots & 0
\end{pmatrix}^T,$$
 which is  a column in the Alexander matrix. 
    Comparing with the one given by the relation $[a,b]=1$ in $G_\Gamma$, they only differs by a multiplication by the nonzero integer $m$. 
    So the Alexander matrix for $G_{\Gamma_\ell}$ can be obtained from the one for $G_\Gamma$ by multiplying the proper (nonzero) weight for the corresponding column. Since we are computing $\sV^1(G_{\Gamma_\ell};\C)$ and any nonzero integer is a unit in $\C$, we get the same fitting ideals for $G_{\Gamma_\ell}$ and $G_\Gamma$, hence the claim follows.
\end{proof}
Now we are ready to prove \cref{thm WRAAG}.

\begin{proof}[Proof of \cref{thm WRAAG}]
  By \cref{example two graphs},  $(ii) \Rightarrow (iii) $ is easy.  
Note that $\langle a_1,a_2 | [a_1,a_2]^m=1 \rangle$ is a K\"ahler group, see e.g. \cite{Ue}. 
Hence so is a finite product of such groups. Then $(iii)\Rightarrow (i)$ follows. 
    We are left to prove $(i)\Rightarrow (ii)$ by two steps. From now on, we assume that $G_{\Gamma_\ell}$ is a K\"ahler group. 
    
    \noindent Step 1: We first show that  $\Gamma_\ell$ has to be a complete graph on an even number of vertices by mimicking the proof presented by
Dimca, Papadima and Suciu in \cite[Theorem 11.7]{DPS}.
    
    Assume that $\Gamma_\ell$ is not a complete graph. Then there exists $W\subseteq V$ such that the subgraph $\Gamma_{W,\ell}$ is maximally disconnected.
Write $W=W_1\sqcup W_2$ with both $W_1 $ and $W_2$ nonempty and no edge connecting $W_1$ and $W_2$. Then  $\Gamma_{W,\ell}=\Gamma_{W_1,\ell}\sqcup \Gamma_{W_2,\ell}$ and so
$G_{\Gamma_{W,\ell}}=G_{\Gamma_{W_1,\ell}}* G_{\Gamma_{W_2,\ell}}.$
By \cite[Proposition 13.3]{PapaSuciu10} we have $\sV^1(G_{\Gamma_{W,\ell}};\C)=(\C^*)^{|W|}$. Moreover, the natural group epimorphism $$G_{\Gamma_{\ell}}\twoheadrightarrow G_{\Gamma_{W,\ell}}$$ 
    gives two embeddings 
   \begin{center}
       $\sV^1(G_{\Gamma_{W,\ell}};\C)\hookrightarrow \sV^1(G_{\Gamma_{\ell}};\C)$
   and $H^1(G_{\Gamma_{W,\ell}};\C)\hookrightarrow H^1(G_{\Gamma_{\ell}};\C) .$
   \end{center} 
By  \cref{cor disconnected}, the image of the first embedding gives an irreducible component of $\sV^1(G_{\Gamma_{\ell}};\C)$. Meanwhile, the image of the 
second embedding is the tangent space of this irreducible component. 
As shown by Dimca, Papadima and Suciu in \cite[Theorem C]{DPS}, such  tangent space for K\"ahler group is a 1-isotropic space, i.e.,  $$H^1(G_{\Gamma_{W,\ell}};\C)\wedge  H^1(G_{\Gamma_{W,\ell}};\C) \to H^2(G_{\Gamma_{\ell}};\C)$$ has a 1-dimensional image and it is a non-degenerate skew-symmetric bilinear form. Since $G_{\Gamma_{W,\ell}}=G_{\Gamma_{W_1,\ell}}* G_{\Gamma_{W_2,\ell}}$, \cite[Lemma 9.4]{DPS} gives that $H^1(G_{\Gamma_{W,\ell}};\C)\wedge  H^1(G_{\Gamma_{W,\ell}};\C) =0$.
Hence we get a contradiction.

    \noindent Step 2: Since $\Gamma_\ell$ is a complete graph,  by \cref{cor disconnected} we get that $\sV^1(G_{\Gamma_\ell};\C)=\{1\}$. 
    Let $X$ be a compact K\"ahler manifold with $\pi_1(X)=G_{\Gamma_\ell}$. 
    By \cref{thm Delzant} and \cref{prop compact orbifold group}, there is no orbifold fibration  $f\colon X\to C_g$ with $g>1$. 
        By  \cref{prop:abelian_bnsr_K\"ahler} and  \cref{thm Delzant}, we have $$\Trop_\Z (\cJ^1(G_{\Gamma_\ell};\Z))=\bigcup_f \mathrm{im} f^* (H^1(C_1;\R) \hookrightarrow H^1(X; \R)) ,$$ 
   where the union runs over all hyperbolic orbifold fibrations $f\colon X\to C_1$ with $C_1$ a compact Riemann surface of  genus 1. Hence it is a finite union of two dimensional real vector spaces. 
    
    For two orbifold fibrations $f\colon X \to C_1$ and $g\colon X \to C'_1$, 
    we claim that either $\mathrm{im} f^* =\mathrm{im} g^* $
    or $  \mathrm{im} f^* \cap  \mathrm{im} g^* =\{0\}$. In fact, since $f$ and $g$ are both holomorphic maps between compact K\"ahler manifolds, both $ H^1(C_1;\R) $ and $H^1(C'_1;\R) $ carry sub-Hodge structure of $H^1(X; \R)$. Then so is $ H^1(C_1;\R)\cap H^1(C'_1;\R) $. Hence  $H^1(C_1;\R)\cap H^1(C'_1;\R) $ has dimension either 0 or 2.

    Assume that $ \Gamma_\ell$ has two adjacent edges $\{a_1,a_2\}$ and $\{a_2,a_3\}$ with weights $m>1$ and $m'>1$, respectively. Pick a primes $p$ such that $p$ divides $m$. Fix an algebraically closed field $\bbmk$ with $\mathrm{char}(\bbmk)=p$.  Let $\Gamma_{1,2}$ denote the subgraph of $G_{\Gamma,\ell} $ with only two vertices $a_1$ and $a_2$ and the edge connecting them of weight $m$. Note that $G_{\Gamma_{1,2}}$ is a compact orbifold group with only one marked point. Then \cref{prop compact orbifold group} gives us that  $ \sV^1( G_{\Gamma_{1,2}};\bbmk)=(\bbmk^*)^2$. The natural group epimorphism $G_{\Gamma_\ell} \twoheadrightarrow  G_{\Gamma_{1,2}}$ induces an embedding $$\sV^1( G_{\Gamma_{1,2}};\bbmk) \hookrightarrow \sV^1(G_{\Gamma_\ell};\bbmk) .$$
   Set  $ \mathbb{R}_{1,2}=\{ (x_1,x_2,0,\cdots,0)\in \R^{|V|} \mid x_1,x_2\in \R\}.$   
 Then we have $$\mathbb{R}_{1,2} \subseteq \Trop_\bbmk (\sV^1(G_{\Gamma_\ell};\bbmk)) \subseteq \Trop_\Z (\cJ^1(G_{\Gamma_\ell};\Z)).$$
 Since we already knew that $\Trop_\Z (\cJ^1(G_{\Gamma_\ell};\Z))$ is a finite union of two-dimensional real vector spaces, $\R_{1,2}$ has to be one of them. 
Hence there exists an orbifold fibration   $f\colon X \to C_1$
  such that $\R_{1,2} =  \mathrm{im} (f^* \colon H^1(C_1;\R) \hookrightarrow H^1(X; \R)) . $
    Similarly, there exists another orbifold fibration   $g\colon X \to C'_1$
  such that $\R_{2,3} =  \mathrm{im} g^* (H^1(C'_1;\R) \hookrightarrow H^1(X; \R)) $ with  $\R_{2,3} = \{ (0,x_2,x_3,0,\cdots,0)\in \R^{|V|} \mid x_2,x_3\in \R\}.$ But $\dim (\R_{1,2}\cap \R_{2,3})=1$ contradicts the previous claim.
\end{proof}

\begin{remark} 
For a right-angled Artin group $G_{\Gamma}$, $\Sigma^1(G_{\Gamma})$ is computed in \cite{MV}. In particular, the first inclusion in (\ref{main inclusion}) holds as equality for $k=1$ (see \cite[Prposition 5.8]{PS06} for a proof). 
It would be interesting to compute $\Sigma^1(G_{\Gamma_\ell})$ for any weighted right-angled Artin group and check if the first inclusion in (\ref{main inclusion}) holds as equality (for $k=1$).
\end{remark}
    
\begin{remark}
Let $\nu \colon  G_{\Gamma_\ell} \to  \Z$ be the homomorphism which sends each generator $a\in V$ to 1.
When $\ell(e)=1$ for all $e\in E$, the kernel of $\nu$ is the Bestvina-Brady group. Dimca, Papadima and Suciu classified the quasi-K\"ahler Bestvina-Brady groups in \cite{DPS08}.  
It would be interesting to know which kernel of $\nu$ for $G_{\Gamma_\ell}$ is quasi-K\"ahler.
\end{remark}


\section*{Acknowledgments}
The authors would like to thank Ziyun He, Laurentiu Maxim, Alexandru I. Suciu and Botong Wang for useful discussions, and thank Thomas Delzant for  \cref{rmk: delzant's cubulable}.  
Yongqiang Liu is supported by National Key Research and Development Project SQ2020YFA070080, the Project of Stable Support for Youth Team in Basic Research Field, CAS (YSBR-001),  the project ``Analysis and Geometry on Bundles'' of Ministry of Science and Technology of the People's Republic of China and  Fundamental Research Funds for the Central Universities. Yuan Liu is partially supported by the China Postdoctoral Science Foundation (No. 2023M744396) and the China Scholarship Council (No. 202406340174).

\bibliographystyle{ssmfalpha} 
\bibliography{biblio}

\end{document}